\documentclass[a4paper,12pt,dvips]{amsart}    
\usepackage{amsmath,amsthm,geometry,graphicx} 
\usepackage{amscd}
\usepackage{psfrag,xspace} 
\usepackage{appendix,fancyhdr}
\usepackage{amssymb}
\usepackage{psfrag,graphicx,color} 
\usepackage{fullpage,url}
\usepackage{amsfonts} 
\usepackage{latexsym} 
\usepackage{epsfig} 
\usepackage[all]{xy}
\usepackage{mathrsfs}
\usepackage{amsthm,amsmath,amssymb,amsfonts, amscd,graphics, latexsym,
enumerate, stmaryrd,xspace,verbatim,epic,eepic}
\usepackage[sectionbib,square]{natbib}
\usepackage[dvips,pagebackref,colorlinks=true,pdftitle={Grothendieck Duality for Deligne-Mumford Stacks},
pdfauthor={Aeshma},pdfpagemode=None]{hyperref}
\xyoption{all}

\numberwithin{equation}{section} 

\theoremstyle{plain}
\newtheorem{thm}{Theorem}[section]

\newtheorem{cor}[thm]{Corollary}
\newtheorem{lem}[thm]{Lemma}
\newtheorem{prop}[thm]{Proposition}
\theoremstyle{definition} 
\newtheorem{defn}[thm]{Definition}

\newtheorem*{ass}{Assumption}

\theoremstyle{remark}
\newtheorem{rem}[thm]{Remark}

\newtheorem{ex}[thm]{Example}

{\nolinebreak $\Box$ \end{trivlist}}

\newcommand{\EE}{\mathcal E}
\newcommand{\FF}{\mathcal F}
\newcommand{\OO}{\mathcal O}

\newcommand{\GG}{\mathcal G}
\newcommand{\XX}{\mathcal X}

\newcommand{\YY}{\mathcal Y}

\DeclareMathOperator{\QCoh}{\text{QCoh}}

\def\cf{\textit{cf.}\kern.3em}

\def\resp{\textit{resp.}\kern.3em}
\renewcommand{\k}{\kern2pt}

\DeclareMathOperator{\supp}{supp}%
\DeclareMathOperator{\Hom}{Hom}%
\DeclareMathOperator{\Ext}{Ext}%
\DeclareMathOperator{\Id}{id}%
\DeclareMathOperator{\Spec}{Spec}%
\DeclareMathOperator{\spec}{Spec}

\DeclareMathAlphabet{\mathpzc}{OT1}{pzc}{m}{it}

\DeclareMathOperator{\HOM}{\mathcal{H}\!\mathpzc{om}}

\DeclareMathOperator{\EXT}{\mathcal{E}\!\mathpzc{xt}}

\DeclareMathOperator{\tr}{tr}
\DeclareMathOperator{\cotr}{cotr}

\renewcommand{\k}{\kern2pt}

\begin{document}

  \title{Grothendieck Duality for  \\  Deligne-Mumford Stacks}
  \author{Fabio Nironi}
  \address{Columbia University, 2990 Broadway, New York, NY 10027}
  \email{fabio.nironi@gmail.com}
 \begin{abstract}
We prove the existence of the dualizing functor for a separated morphism of algebraic stacks with affine diagonal; then we  explicitly develop duality for compact  Deligne-Mumford stacks focusing in particular on the morphism from a stack to its coarse moduli space and on representable morphisms. We explicitly compute the dualizing complex for a smooth stack over an algebraically closed field and prove that Serre duality holds for smooth compact Deligne-Mumford  stacks in its usual form. We prove also that a proper Cohen-Macaulay stack  has a dualizing sheaf and  it is an invertible sheaf when  it is Gorenstein. As an application of this general machinery we compute the dualizing sheaf of a tame nodal curve.     
\end{abstract}

\maketitle 

\tableofcontents

\section*{Overview}

The first part of the work is foundational, it deals with the existence of the dualizing complex through the abstract machinery developed by Deligne in \cite[Appendix]{MR0222093} and refined by Neeman in \cite{MR1308405}. Let $\pi\colon\XX\to X$ be a projective stack (tame separated global quotient with projective moduli scheme), it is possible to apply Neeman's approach to duality along the morphism $\pi$. We can provide a very fast proof of duality in $D(\XX)=D(\text{QCoh}(\XX))$ using that $\pi_\ast$ is exact on quasicoherent sheaves and exhibiting a generating set of $D(\XX)$ using a $\pi$-very ample vector bundle on $\XX$ (a generating sheaf according to \cite{MR2007396}).
It is also possible to prove duality, in a not so fast way, following the original strategy of Deligne. We first reduce the problem, of finding an adjoint of the derived push-forward, from the triangulated category $D^\ast(\XX)$ to the category of complexes; then we apply a representability result in the category of complexes and go back to derived category. With this approach we can remove the unnecessary hypothesis that a $\pi$-very ample vector bundle exists and we can also study more general morphisms. In general the result we obtain holds in $D^+(\XX)$, but we can prove that when the stack $\XX$ has a good moduli space (see \cite[Def 4.1]{alper-2008} for the definition) the derived push-forward is defined in $D(\XX)$ and so is its adjoint. In this part of the work we consider algebraic stacks that are quasi-compact and with affine diagonal, however all the explicit results that follow the first section are proved only for noetherian separated Deligne-Mumford stacks.

In the second section we study the compatibility of the dualizing functor with flat base change, we first prove that duality for stacks is \'etale local, and then we reconstruct Verdier's result in \cite{MR0274464} for Deligne-Mumford stacks. Using the compatibility with flat base change we are able to prove Serre Duality for  Deligne-Mumford stacks smooth and proper over an algebraically closed field and duality for finite morphisms. We obtain all the expected results: the dualizing sheaf for a smooth proper stack is the canonical bundle shifted  by the dimension of the stack, for a closed embedding $i\colon \XX \to \YY$ in a smooth proper stack $\YY$ the dualizing complex of $\XX$ is $\EXT_\YY^\bullet(\OO_\XX,\omega_\YY)$ where $\omega_\YY$ is the canonical bundle. This is a coherent sheaf if $\XX$ is Cohen-Macaulay, an invertible sheaf if it is Gorenstein.  

In the last part of the work  we use this abstract machinery to compute the dualizing sheaf of a tame nodal curve. We prove that the dualizing sheaf of a curve without smooth orbifold points is just the pullback of the dualizing sheaf of its moduli space. Smooth orbifold points give a non trivial contribution that can be computed using the root construction (Cadman \cite{Cstc-2007}, Abramovich-Graber-Vistoli \cite{AGVgwdms}).  We compute also the dualizing sheaf of a local complete intersection proving that it is the determinant of the cotangent complex shifted by the dimension of the stack, as it is in the scheme-theoretic setup.  

This paper has been partially motivated by our study on semistable sheaves on projective stacks in \cite{stack-stab}. In particular we have used Grothendieck duality to handle the definition of dual sheaf in the case of sheaves of non maximal dimension. Given a $d$-dimensional sheaf $\FF$  on a projective Cohen-Macaulay stack $p\colon\XX\to\Spec{k}$ over $k$ an algebraically closed field, the dual $\FF^D$ is defined to be $R\HOM_\XX(\FF,p^!k)$ (as usual). If the sheaf $\FF$ is torsion free on a smooth stack this is just $\FF^\vee\otimes\omega_\XX$ where $\omega_\XX$ is the canonical bundle. Using Grothendieck duality we can prove that there is a natural morphism $\FF\to\FF^{DD}$ which is injective if and only if the sheaf is pure. We use this basic result in the GIT study of the moduli scheme of semistable pure sheaves \cite[Lem 6.10]{stack-stab}. 

\section*{Acknowledgements}

I am mainly indebted to Andrew Kresch who assisted me during this study, pointing out tons of mistakes and guiding me with his sharp criticism when I tended to be sloppy. I am thankful to the Z\"urich  Institut f\"ur Mathematik for hospitality during  February 2008 and for providing me with an office provided with a couch. I wish to thank Daniel Hern\'andez Ruip\'erez who encouraged me to write a better proof of existence following Deligne's version, and who actually assisted and helped me during the realization of the work, and of course I wish to thank the university of Salamanca for hospitality.
I also want to thank Ulrich Bunke for pointing out an interesting example where duality is possibly not working in unbounded derived category; this example shaded a new light on the problem showing that the sufficient hypothesis to prove duality in unbounded derived category was the existence of a good moduli space of the source stack or the morphism to be cohomologically affine.
I am grateful to Joseph Lipman for the interest he showed into my work and for providing numerous helpful suggestions and critics. 
I also say thanks to Elena Andreini, Barbara Fantechi, Etienne Mann and Angelo Vistoli for many helpful discussions.   
  
\section{Foundation of duality for stacks}

\subsection{History}
Serre duality for stacks can be easily proven with some ad hoc argument in specific examples, such as orbifold curves, gerbes, toric stacks and others; however a general enough proof requires some abstract machineries. Hartshorne's approach in Residues and Duality \cite{MR0222093} is not suitable to be generalized to algebraic stacks (not in an easy way at least). In the appendix of the same book Deligne proves (in a very concise and elegant way) the following statement:
\begin{lem}
  Let $X$ be a quasi-compact scheme (non necessarily noetherian) and $\text{QCoh}_X$ the category of quasi coherent sheaves on $X$. Let $F\colon\text{QCoh}_X^\circ\to \mathfrak{Set}$ be a left exact contravariant functor sending filtered colimits to filtered limits, then the functor $F$ is representable.  
\end{lem}
Using this statement it's easy to prove the following:
\begin{thm}
  Let $p\colon X\to Y$ be a morphism of  separated noetherian schemes, $F$ a sheaf on $X$ and $C^\bullet(F)$ a \textit{functorial} resolution of $F$ acyclic with respect to $p_\ast$. Moreover let $G$ be a quasi coherent sheaf on $Y$ and $I^\bullet$ an injective resolution of $G$:
  \begin{enumerate}
  \item the functor $\Hom_Y(p_\ast C^q(F),I^p)$ is representable for every $q,p$ and represented by a quasi coherent injective $p_q^{!}I^p$.
  \item the injective quasicoherent double-complex $p_q^{!}I^p$ defines a functor $p^!\colon D^+(Y)\to D^+(X)$ which is right adjoint of $Rp_\ast$:
    \begin{displaymath}
      R\Hom_Y(Rp_\ast F,G)\cong R\Hom_X(F,p^!G)
    \end{displaymath}
  \end{enumerate}
\end{thm}

This same strategy has been applied with success by Hern\'andez Ruip\'erez  in \cite{RUIP81} to prove duality for algebraic spaces.  

In \cite{MR1308405} Neeman proved Grothendieck duality for schemes using Brown representability theorem and Bousfield localization. Instead of using a representability result in the category of complexes he uses a representability result directly in the derived category.
\begin{thm}
  Let $\mathcal{T}$ be a triangulated category which is compactly generated and $H\colon\mathcal{T}^\circ\to\mathfrak{Ab}$ be a homological functor. If the natural map:
  \begin{displaymath}
    H\Bigl(\coprod_{\lambda\in\Lambda} x_\lambda\Bigr)\to \prod_{\lambda\in\Lambda} H(x_\lambda) 
  \end{displaymath}
is an isomorphism for every small coproduct in $\mathcal{T}$, then $H$ is representable.
\end{thm}
If a scheme has an ample line bundle then it's easy to prove that $D(X)$ is compactly generated. Every scheme admits locally an ample line bundle (take an affine cover then the structure sheaf is ample); to verify that local implies global Bousfield localization is used.

In the case of stacks we can  prove that if $\XX$ has a generating sheaf and $X$ an ample invertible sheaf then $D(\XX)$ is compactly generated, so that we can use Brown representability. It's true again that every tame stack $\XX$  has \textit{\'etale} locally a generating sheaf \cite[Prop 5.2]{MR2007396}, however the argument used by Neeman to prove that local implies global heavily relies on Zariski topology and cannot be generalized to stacks in an evident way.

This kind of approach is very very fast but probably difficult to be improved, moreover the existence of the generating sheaf, even if not so restrictive, is a very unnatural hypothesis. In order to remove this hypothesis we go back to Deligne's approach and adapt it to the case of algebraic stacks. Eventually we  prove existence of a dualizing functor for a separated quasi-compact morphism of quasi-compact  algebraic stacks with affine diagonal. While the previous result holds in unbounded derived category, this one produces a dualizing functor defined on derived category bounded from below. 

Once existence and uniqueness are proved, we will be able to determine the actual shape of the dualizing functor in many examples modifying Verdier's approach \cite{MR0274464} to the computation of duality. 


\subsection{Existence using Neeman's technique}
\begin{ass}
  In this section every stack and every scheme is quasi-compact. Unfortunately the word \textit{generating} appears in this paper with three different meanings: we have generating sheaves that are locally free sheaves very ample with respect to the morphism from a stack to the moduli space, and we have two different notions of generating set. In this section a generating set is a generating set of a triangulated category in the sense of \cite[Def 1.8]{MR1308405}.
\end{ass}
We start with the fast approach to duality using Neeman's results.
We will denote with $D(\XX)=D(\text{QCoh}(\XX))$ the derived category of quasicoherent sheaves on $\XX$. 

\begin{lem}
  Let $\pi\colon\XX\to X$ be a tame separated stack with moduli space $X$. The functor $\pi_\ast\colon D(\XX)\to D(X)$ respects small coproducts, that is the natural morphism:
  \begin{equation}\label{eq:8}
    \coprod_{\lambda\in\Lambda}\pi_\ast x_\lambda \to \pi_\ast \coprod_{\lambda\in\Lambda} x_\lambda
  \end{equation}
is an isomorphism for every small $\Lambda$.
\end{lem}
\begin{proof}
  We recall that the category of sheaves of modules on a site satisfies  $AB4$; the reason is that the category of modules satisfies $AB4$ which implies that presheaves satisfy $AB4$ and we can conclude using \cite[Thm II.2.15.a-d]{Met}. Arbitrary  coproducts are left exact, and as a matter of fact exact.
We choose a smooth presentation  $X_0\to\XX$ and we associate to it  the simplicial nerve $X_\bullet$. Let $f^i\colon X_i\to X$ be the obvious composition. For every quasicoherent sheaf $\mathcal{F}$ on $\XX$ represented by $\mathcal{F}_\bullet$ on $X_\bullet$ we have an exact sequence:
\begin{equation}\label{eq:the-exact-sequence}
  0\to\pi_\ast\mathcal{F}\to f^0_\ast\mathcal{F}_0\to f^1_\ast\mathcal{F}_1
\end{equation}
 The result follows from left exactness of the coproduct, the analogous result for schemes \cite[Lem 1.4]{MR1308405} and the existence of the natural arrow (\ref{eq:8}).
\end{proof}
\begin{cor}\label{cor:commwithcoprod2}
  Let $\pi\colon\XX\to X$ be as in the previous statement and $f\colon\XX\to Y$ be a quasi-compact separated morphism to a scheme $Y$. The functor $R f_\ast\colon D(\XX)\to D(Y)$ respects small coproducts.
\end{cor}
\begin{proof}
  This is an immediate consequence of the universal property of the moduli space $X$ and the previous lemma.
\end{proof}

 Let  $\EE$ be a generating sheaf of $\XX$ and $L$ an ample invertible sheaf of $X$. We call the couple ($\EE$,$L$) a polarization.
\begin{lem}\label{lem:compactlygen}
  Let the stack $\pi\colon\XX\to X$ be tame and  the sheaves $\EE$ and $L$ are a polarization. The derived category $D(\XX)$ is compactly generated and the set $T=\{\EE\otimes\pi^\ast L^n [m]\; \vert\; m,n\in\mathbb{Z}\}$ is a generating set.  
\end{lem}
\begin{proof}
  Same proof as in \cite[Ex 1.10]{MR1308405}, but using that every quasi coherent sheaf $\FF$ on $\XX$ can be written as a quotient of $\EE^{\oplus t}\otimes\pi^\ast L^n$ for large enough  integers  $t,n$.
\end{proof}
\begin{rem}
  The most important class of algebraic stacks $\XX$ satisfying conditions in the previous lemma is composed by projective stacks and more generally families of projective stacks; the second class we have in mind is given by stacks of the kind $[\Spec{B}/G]\to\Spec{A}$ where $G$ is a linearly reductive group scheme on $\Spec{A}$, which is the structure of a tame stack \'etale locally on its moduli space.
\end{rem}
\begin{prop}
  Let $\pi\colon\XX\to X$ be a tame stack with a polarization $\EE$,$L$. Let $f\colon \XX\to Y$ be a separated quasi-compact morphism to a scheme.  The functor $R f_\ast\colon D(\XX)\to D(X)$ has a right adjoint  $f^!\colon D(\YY)\to D(\XX)$. 
\end{prop}
\begin{proof}
  This is a formal consequence of Brown representability \cite[Thm 4.1]{MR1308405}, Lemma \ref{lem:compactlygen} and Corollary \ref{cor:commwithcoprod2}.
\end{proof}

\subsection{Existence for algebraic stacks with affine diagonal}

\begin{ass}
 We fix a base scheme $S$ which is separated. Every stack is algebraic quasi-compact  and with affine diagonal. The reason why we ask the diagonal to be affine is that we want a morphism from an affine scheme to a stack to be an affine morphism.  In the case of Deligne-Mumford stacks separated implies finite diagonal and in particular affine,  and affine diagonal is a weaker notion then separated as in the case of schemes. In the case of Artin stacks separated doesn't imply  affine diagonal so that the two concepts are quite unrelated. The class of Artin stacks with affine diagonal includes stacks that are not global quotients (\cite[Ex 3.12]{MR1844577}). A good class of Artin stacks satisfying these assumptions is the class of separated tame stacks.

We fix a universe $\mathcal{U}$ such that  lisse-\'etale sites of stacks we are using are $\mathcal{U}$-sites. The word \textit{small} means a set in this universe.

We will denote with $D(\XX)$ the derived category of quasicoherent complexes on an algebraic stack $\XX$. Every argument we use requires only quasicoherent acyclic resolutions, in this way we can avoid the use of $D_{qc}(\XX)$ (derived category of complexes with quasicoherent cohomology) which is definitely more slippery in the case of Artin stacks. 

In this section and from now on a generating set is a generating set of an abelian category as defined for instance in \cite[V.7 pg. 123]{MR0354798}.
\end{ass}

\begin{lem}
Let $\XX$ be an algebraic $S$-stack quasi-compact with affine diagonal, every morphism from an affine scheme to $\XX$ is affine.
\end{lem}
\begin{proof}
  Let $\spec{A}\to S$ be an affine $S$-scheme. Since $S$ is separated this morphism is affine. Choose $p_0\colon X_0\to \XX$ an affine atlas. Let $f\colon \spec{A}\to \XX$ be an $S$-morphism and we want to prove it is an affine morphism. It's an easy check that the two squares in the following diagram are $2$-cartesian:
  \begin{displaymath}
    \xymatrix{
      \spec{A}\times_{\XX}X_0 \ar[r]^-{\delta}\ar[d] &  \spec{A}\times_S X_0\ar[d] \ar[r]^-{q_1} &  X_0\ar@{->>}[d] \\
      \spec{A} \ar[r]^-{\Delta} &  \spec{A}\times_S\XX \ar[r]^-{p_1} &  \XX \\
}
  \end{displaymath}
The composition $p_1\circ\Delta$ in the diagram is $f$. The morphism $\delta$ is affine because the diagonal is affine and $q_1$ is affine because $\spec{A}\to S$ is affine. We can conclude that $f$ is affine using faithfully-flat descent.
\end{proof}
\begin{lem}\label{lem:commute-with-filt}
  Let $f\colon\XX\to\YY$ be a separated quasi-compact morphism of algebraic stacks. The functor $f_\ast\colon\text{QCoh}(\XX)\to\text{QCoh}(\YY)$ commutes with filtered colimits.
\end{lem}
\begin{proof}
  It follows from the analogous statement  for schemes using \cite[Lem 12.6.2]{LMBca}.
\end{proof}
\begin{lem}\label{lem:inject-qcoh}
  Let $\XX$ be a quasi-compact algebraic stack and $F$ a cartesian sheaf. Let $p_0\colon X_0\to \XX$ be an affine atlas. The natural morphism $F\to {p_0}_\ast p_0^\ast F$ is injective.
\end{lem}
\begin{proof}
  In the case of Deligne-Mumford stacks we can choose $p_0$ to be \'etale and the result follows immediately from the definition of sheaf, but if we work with the lisse-\'etale site there is actually something to prove. The map is injective if and only if it is injective on every object $U\to\XX$ of the lisse-\'etale site. Denote with $U_0$ the pullback of $X_0$ to $U$ with $\overline{p}\colon U_0\to U$ the pullback of $p_0$ and with $F_U$ the restriction of $F$ to $U$. We have to prove that $F_U\to \overline{p}_\ast\overline{p}^\ast F_U$ is injective. Unfortunately the map $\overline{p}$ is not a covering in lisse-\'etale so the result doesn't follow from the definition of sheaf. We can fix the problem using the trick in \cite[Ex 2.51]{MR2223406}. According to \cite[Cor 17.16.3]{MR0238860}  we can find an \'etale cover $V\to U$ factorizing through $U_0$. We can pullback the map $\overline{p}$ to an \'etale map $\overline{q}\colon V_0\to V$. In this way we have produced an \'etale cover of $U$ that refines the smooth cover $U_0$. The problem is now reduced to the \'etale situation using faithfully flatness of the cover $V\to U$. 
\end{proof}
\begin{lem}\label{lem:small-gen}
  Let $\XX$ be a quasi-compact algebraic stack. The set of locally finitely presented sheaves (up to isomorphisms) on $\XX$ is a small generating set for the category of quasicoherent sheaves.
\end{lem}
\begin{proof}
  First of all we prove that locally finitely presented sheaves are a small set.  The proof is a simple variation of a classical argument of Kleiman \cite[Thm 4]{MR578050} or dually Lipman \cite[4.1.3.1]{lipman1960}. Let $\XX$ be a quasi-compact algebraic stack. The set of locally finitely presented sheaves (up to isomorphisms) on $\XX$ is a small generating set for the category of quasicoherent sheaves on $\XX$. First of all we observe that sheaves of a $U$-topos are always well-powered (and co-well-powered) \cite[Cor 0.2.7.vii]{Gcna-1971}, it means that for every sheaf in the topos the set of every subobject (or quotient) up to isomorphisms is a small set. Let $p_0:\coprod_i U_i\to \XX$ be an affine atlas. According to Lemma \ref{lem:inject-qcoh} every quasicoherent sheaf $F$ is isomorphic to a subobject of ${p_0}_\ast p_0^\ast F$. Assume that $F$ is locally finitely presented, than $p_0^\ast F$ is also locally finitely presented. for this reason we can state that $p_0^\ast F$ is a quotient  of $\bigoplus_i \OO_{U_i}^{\oplus n_i}$ for integers $n_i$. Sheaves of this kind are numerable and using that sheaves on $\XX$ are well-powered we can conclude that locally finitely presented are a small set. To prove that locally finitely presented are a generating set for quasicoherent sheaves we have just to observe that every quasicoherent sheaf is a filtered colimit of its locally finitely presented subobjects (same proof as in \cite[15.4]{LMBca} with some modification). 
\end{proof}
\begin{prop}
  Let $\XX$ be an algebraic stack, the category of quasicoherent sheaves on $\XX$ has enough injectives and $K$-injectives.
\end{prop}
\begin{proof}
  The category of quasicoherent sheaves on the lisse-\'etale site is cocomplete \cite[Lem 13.2.5]{LMBca}, satisfies $AB5$ and has a small generating set according to Lemma \ref{lem:small-gen} which implies it is a Grothendieck category and in particular it has enough injectives and $K$-injectives \cite[Thm 5.4]{alonso}. 
\end{proof}
For the definition of a $K$-injective we refer the reader to the paper where they have been introduced \cite[Def pg 124]{MR932640}.
At this point the only missing ingredient is a functorial quasicoherent resolution that is $f_\ast$-acyclic and commutes with filtered colimits. We achieve the first requirement using a variation of the \v{C}ech resolution. Let $\XX$ be an algebraic stack. We choose a smooth affine atlas $p_0\colon X_0\to\XX$, which is possible since $\XX$ is quasi-compact. Given a quasicoherent sheaf $\FF$ we associate to it ${p_0}_\ast p_0^\ast\FF$. Having assumed that $\XX$ has affine diagonal the morphism $p_0$ is affine and the functor ${p_0}_\ast p_0^\ast$ is exact. We define $\mathcal{A}$ to be the set of quasi coherent sheaves ${p_0}_\ast p_0^\ast\FF$ where $\FF$ is quasicoherent. It follows from Lemma \ref{lem:inject-qcoh} that every $\FF$ has a natural injection $\FF\to{p_0}_\ast p_0^\ast\FF $ so we can produce a resolution of $\FF$ in this way: define $K_0$ to be $\FF$ and the sheaf $K_i$ the cokernel $0\to K_{i-1}\to {p_0}_\ast p_0^\ast K_{i-1} \to K_i\to 0$. The resulting resolution is:
\begin{equation} \label{eq:cech-like-res}
0\to \FF\to {p_0}_\ast p_0^\ast\FF\to {p_0}_\ast p_0^\ast K_1 \to  {p_0}_\ast p_0^\ast K_2 \ldots
\end{equation}
We will denote this resolution of $\FF$ with $C^\bullet(F)$
\begin{lem}
  The objects in $\mathcal{A}\subset\QCoh(\XX)$ are $f_\ast$-acyclic, the set $\mathcal{A}$ is closed under finite direct sums and every object in $\QCoh(\XX)$ can be injected in an object of $\mathcal{A}$. Resolutions with objects in $\mathcal{A}$ can be used to compute $Rf_\ast: D^+(\XX)\to D^+(\YY)$. 
\end{lem}
\begin{proof}
Since ${p_0}_\ast$ has an exact left adjoint it maps injectives to injectives. Given ${p_0}_\ast p_0^\ast F$ we can compute its higher derived images using the theorem on the composition of derived functors: 
\begin{displaymath}
R^if_\ast({p_0}_\ast p_0^\ast F)=R^i(f\circ{p_0})_\ast p_0^\ast F= 0\quad \text{for $i>0$}
\end{displaymath}
where we have used that ${p_0}_\ast$ is exact and that $(f\circ{p_0})_\ast$ is exact.
\end{proof}
\begin{rem}
\begin{enumerate}
\item It follows from Lemma \ref{lem:commute-with-filt} that the functor ${p_0}_\ast p_0^\ast$ commutes with filtered colimits.
\item Resolutions with objects in $\mathcal{A}$ are  $f_\ast$-acyclic and functorial but in general  not finite; for this reason we  have to work in bounded below derived category. The situation cannot be improved at this level of generality since we do not expect an Artin stack to have finite cohomological dimension so that the derived push-forward is defined in bounded below derived category. We will improve the result when the morphism $f$ is cohomologically affine (\cite[Def 3.1]{alper-2008} it is quasi-compact and the push-forward is exact on quasicoherent sheaves) or when the stack $\XX$ has a good moduli space (\cite[Def 4.1]{alper-2008} the morphism $\pi\colon \XX\to X$ is cohomologically affine and $\pi_\ast\OO_\XX\cong \OO_X $). 
\end{enumerate}
\end{rem}
\begin{thm}\label{thm:deligne-rep}
  Let $f\colon\XX\to\YY$ be a separated quasi-compact morphism of algebraic stacks. The functor $Rf_\ast: D^+(\XX)\to D^+(\YY)$ has a right adjoint $f^!\colon D^{+}(\YY)\to D^+(\XX)$. 
\end{thm}
\begin{proof}
  The proof is a standard argument of Deligne \cite[Appendix pg 416]{MR0222093}. For the convenience of the reader we give an outline of the argument.
Let $F$ be an object in $D^+(\XX)$ and $G$ in $D^+(\YY)$. We consider the functor $R\Hom_\YY(Rf_\ast F,G)$ and we want to produce a right adjoint $f^!G$. We replace $G$ with $I^\bullet$ a complex of injectives and $F$ with a quasi-isomorphic complex produced using the functorial resolution $C^\bullet(\cdot)$. We will denote also this complex with $C^\bullet(F)$.  We have reduced the problem to the category of complexes. We want to prove that the functor $\Hom_\YY(f_\ast C^p(\cdot),I^q)$ from complexes of quasicoherent sheaves to the category of abelian groups is representable for every $p,q\in\mathbb{Z}$. The functor $f_\ast$ commutes with filtered colimits according to Lemma \ref{lem:commute-with-filt}, the resolution $C^\bullet$ commutes with filtered colimits for the same reason,  sheaves of an $U$-topos have small $\Hom$-sets (an $U$-topos is an $U$-category so it has small $\Hom$-sets\footnotemark \footnotetext{I am quoting Giraud and not SGA4 just because his book contains in the first pages a very short summary about $U$-topoi and it is very easy to consult} \cite[0.2.6.iv]{Gcna-1971}) and according to Lemma \ref{lem:small-gen} the category of quasicoherent sheaves has a small generating set so we can apply the special adjoint functor (\cite[Thm V.8.2]{MR0354798} with reverted arrows) and conclude that $\Hom_\YY(f_\ast C^p(\cdot),I^q)$ is represented by a quasicoherent sheaf $f^{-p}I^q$. Since the functor is also exact we can conclude that $f^{-p}I^q$ is injective and actually a double complex of injectives in the first quadrant. The total complex of $f^{-p}I^q$ is quasi-isomorphic to $f^!G$.
\end{proof}
\begin{thm}\label{thm:existence-adj-unbound}
  Let $f\colon\XX\to\YY$ a separated quasi-compact morphism of algebraic stacks and assume that $f$ is cohomologically affine or $\XX$ has a good moduli space which is a quasi-compact separated scheme.  The functor $Rf_\ast: D(\XX)\to D(\YY)$ has a right adjoint $f^!\colon D(\YY)\to D(\XX)$.
\end{thm}
\begin{lem}
Under the assumptions of the previous theorem the functor $Rf_\ast$ has finite cohomological dimension. 
\end{lem}
\begin{proof}
  If $f$ is cohomologically affine $f_\ast$ is exact and there is nothing to prove.  Assume $\XX$ has a good moduli space which is a scheme $\pi\colon \XX\to X$.  We want to prove that there is $n\in\mathbb{Z}$ such that for every quasicoherent sheaf $F$ on $\XX$ the sheaf $R^if_\ast F=0$ for every $i>n$. Since $R^if_\ast F$ is quasicoherent it is enough to prove that $R^if_\ast F\vert_{Y_0}=0$ for a smooth cover $Y_0\to\YY$ made of a finite number of affine schemes (again we are using that $\YY$ is quasi-compact). Denote with $U$ an affine scheme of the cover and $\XX_U$ the base change of $\XX$ to $U$. Since the morphism $\XX_U\to\XX$ is affine the morphism $\XX_U\to X$ is cohomologically affine (\cite[Prop 3.9.i-ii]{alper-2008}). According to \cite[Rem 4.3]{alper-2008} $\XX_U$ has a good moduli space $X_U$ which has a unique  morphism to $X$ making the diagram commute:
  \begin{displaymath}
    \xymatrix{
      \XX_U \ar[r]\ar[d]^{\pi_U} &  \XX\ar[d]^{\pi} \\
      X_U \ar[r]^f & X \\
}
  \end{displaymath}
and moreover the morphism $f$ is affine. From the universal property of a good moduli space we have also an arrow $g\colon X_U \to U$ making the following diagram commute:
\begin{displaymath}
  \xymatrix{
    \XX_U \ar[dd] \ar[rd]^{\pi_U} & \\
    & X_U\ar[ld]_g \\
    U & \\
}
\end{displaymath}
The sheaf $R^if_\ast F\vert_U$ is just the $\OO_U$-module $H^i(\XX_U,F\vert_{\XX_U})$ which is just $H^i(X_U,{\pi_U}_\ast (F\vert_{\XX_U}))$. The morphism $X_U\to X$ is affine and $X$ is separated and quasi-compact, so we can conclude that $X_U$ is separated and quasi compact. Since $X_U$ is separated $g$ is separated, moreover $X_U$ is quasi-compact and $U\to S$ is separated so we deduce that $g$ is quasi-compact. The morphism $g$ is separated and quasi-compact (and of course $U$ itself is quasi-compact) so we can conclude that there is $n$ depending on $U$ and not on $\FF$ such that cohomology groups vanish for $i>n$. Since affine schemes $U$ are in finite number we can choose a finite $n$ satisfying the requirements of the lemma.  
\end{proof}
\begin{proof}[proof of theorem \ref{thm:existence-adj-unbound}]
  It is an easy variation of Theorem \ref{thm:deligne-rep}. Instead of taking an injective complex $I^\bullet$ bounded from below we take a $K$-injective complex not necessarily bounded and we observe that if $I$ is $K$-injective then $f^{-p}I$ is $K$-injective. Instead of using the infinite functorial resolution $C^\bullet(F)$ we truncate it at stage $d$ where $d$ is given by the cohomological dimension, which is finite according to the previous lemma. It is a standard fact that the last term in the finite resolution is again acyclic (\cite[Prop 2.7.5.iii]{lipman1960} with inverted arrows), it is obviously functorial, and it is an easy check that it is exact and preserves filtered colimits, then we can conclude the proof as before. 
\end{proof}

\begin{rem}
\begin{enumerate}
\item We don't know if the existence of a good moduli space is actually a necessary condition to work in unbounded derived category,  at least it is quite relevant: if we remove this hypothesis it is possible to produce examples where the derived push forward in unbounded derived category doesn't commute with arbitrary coproducts. A very simple example has been shown us by Ulrich Bunke: it is enough to take $\XX$ a non tame trivial gerbe over a point and $\YY$ its moduli scheme (which is not good) and test the statement on an unbounded complex like $\coprod_{n\in\mathbb{Z}} F[n]$ where $F$ is a coherent sheaf on the gerbe which carries a trivial representation.
\item Since the dualizing functor is constructed as an adjoint (in $D$ or $D^+$) it is unique in a suitable sense. Moreover the dualizing functor is compatible with the composition of morphisms.
 Assume we have a composition $g\circ f\colon \XX\to\YY\to\mathcal{Z}$ such that both $Rg_\ast$  and $Rf_\ast$ have a right adjoint. The two functors $Rg_\ast Rf_\ast$ and $R(g\circ f)_\ast$ are canonically isomorphic and duality gives us a canonical isomorphism:
\begin{equation}
  \label{eq:iso-compos-upper-shriek}
\xymatrix{
  f^! g^! \ar[r]^-{\eta_{f,g}} & (g\circ f)^! \\
}
\end{equation}
\end{enumerate}
\end{rem}

In order to compute duality in real situations it is usefull to know the behaviour of the dualizing functor with respect to the tensor product. The following result gives use the opportunity to reduce the computation of $f^!F$ to the computation of $f^!\OO_\YY$.
\begin{prop}
  Let $f\colon\XX\to \YY$ be a separated finite-Tor dimensional morphism of noetherian algebraic stacks, assume that $\YY$ has the resolution property (see \cite{MR2108211} for a comprehensive study of the problem), and both $\XX$ and $\YY$ have affine diagonal. Let $F$ be an object in $D^b(\text{Coh}(\YY))$ (or $D^-(Coh{\YY})$ when $f$ is cohomologically affine or $\XX$ has a good moduli space). There is a natural isomorphism:
  \begin{displaymath}
    Lf^\ast F\overset{L}{\otimes} f^! \OO_\YY \to f^!F
  \end{displaymath}
\end{prop}
\begin{proof}
As far as we can work with complexes of locally free sheaves the proof is the same argument used in  \cite[Thm 5.4]{MR1308405} for compact objects (perfect complexes in the case of schemes).  
\end{proof}
\begin{rem}
 Unfortunately we cannot conclude the proof as in \cite[Thm 5.4]{MR1308405} stating that the result for  compact objects implies the result for every $F\in D(\YY)$.

\end{rem}

\begin{defn}
 Let $f\colon\XX\to S$ be an $S$-stack, $f$ the structure morphism. Suppose $f^!$ exists, we will call $f^!\OO_S$ the dualizing complex of $\XX$.
\end{defn} 

To conclude we want to spend a few words about the Godement resolution for stacks and the possibility to use it to compute duality. In the case of schemes it is possible to use the Godement resolution to compute duality in $D_{qc}(\XX)$ (\cite{lipman1960} Chapter $4.1$). The Godement resolution is functorial exact acyclic but it doesn't commute with filtered colimits since it is defined through an infinite product. However it is possible to ``modify it'' in order to acquire this property too, and this has been done by Lipman in \cite{lipman1960} following a vague suggestion of Deligne in \cite[Appendix pg 417]{MR0222093}. In the case of stacks it is still possible to define the Godement resolution both on the \'etale site and the lisse-\'etale site. Unfortunately the Godement resolution on the lisse \'etale site is not made of cartesian sheaves! As far as we work with Deligne-Mumford stacks it is possible to reproduce Lipman's proof \cite[Thm 4.1]{lipman1960} under the same hypothesis. In the case of Artin stacks actually we don't know. Being non cartesian is not a big problem per se, since  in $D_{qc}(\XX)$ complexes are allowed to be non cartesian; however the standard theorem on the push-forward \cite[Lem 12.6.2]{LMBca} fails and we are not even able to prove that the Godement resolution is $f_\ast$-acyclic.

\subsection{Proper representable morphisms of algebraic stacks}

In order to explicitly compute Serre duality for a smooth proper stack  we need a more explicit approach to duality when the morphism is representable. A result analogous to the proposition we are going to prove is contained in \cite[Thm 25.2]{lipman1960}. A similar problem  for proper representable morphisms of locally compact topological stacks has been addressed and solved in \cite[Section 6]{bunke-2008}.
Let $f\colon\YY\to\XX$ be a representable proper morphism and choose a smooth presentation $\xymatrix{X_1 \ar@<1ex>[r]^-s \ar@<-1ex>[r]^-t & X_0\ar[r]^-{p_0} & \XX}$ and produce the pullback presentation  $\xymatrix{Y_1 \ar@<1ex>[r]^-u \ar@<-1ex>[r]^-v & Y_0\ar[r]^-{p_0} & \YY}$ . Let $I^\bullet$ be a bounded below complex of injectives on $\XX$ and $\alpha_p\colon s^\ast I_0^p\to t^\ast I_0^p$ the isomorphism defining the equivariant structure of the sheaf $I_0^p$. Using flat base change theorem for schemes \cite[Thm 2]{MR0274464} we can produce the following chain of isomorphisms:
\begin{displaymath}
\xymatrix{
 u^{\ast} f_0^{!} I_0^p \ar[r]^-{c_s} & f_1^{!} s^{\ast} I_0^p\ar[r]^-{f_1^{!}\alpha_p} & f_1^{!} t^{\ast} I_0^p & \ar[l]_-{c_t} v^{\ast} f_0^{!} I_0^p  \\
}
\end{displaymath}
Call this isomorphism $\beta_p$.  It satisfies the cocycle condition because $\alpha_p$ does and the isomorphisms $c_t,c_s$ defined in \cite[pg 401]{MR0274464} satisfy it according to the same reference (see next section for a recall of the construction of $c_s$). The data $\beta_p, f_0^! I_0^p$ define a complex of injectives  on $\YY$. Running through the values of $p$ we obtain a double complex of injectives  and we  denote with $f^{\natural} I^\bullet$ the total complex associated to it.
\begin{prop}\label{prop:representable-proper-duality}
   Let $f\colon\YY\to\XX$ be a morphism as above. Let $\FF^\bullet\in D^{+}(\XX)$ and $I^\bullet$ an injective complex quasi-isomorphic to it. The derived functor  $f^!\FF^\bullet$ is computed by $f^{\natural} I^\bullet$. 
\end{prop}
\begin{proof}
Let $J$ be an injective sheaf on $\YY$ and $I$ an injective sheaf on $\XX$. Keeping notations introduced above we can write the following exact sequence:
\begin{displaymath}
  \xymatrix{
    0 \ar[r] & \HOM_{\XX}(f_\ast J,I) \ar[r] & {p_0}_\ast\HOM_{X_0}(p_0^\ast f_\ast J,p_0^\ast I) \ar[r] & {p_0}_\ast s_\ast\HOM_{X_1}(s^\ast p_0^\ast f_\ast J, s^\ast p_0^\ast I)
}
\end{displaymath}
Using duality for proper morphisms of schemes and flat base change for the twisted inverse image we have the following commutative square:
\begin{displaymath}
  \xymatrix{
   {p_0}_\ast\HOM_{X_0}(p_0^\ast f_\ast J,p_0^\ast I) \ar[r]\ar[d]^-{\wr} &  {p_0}_\ast s_\ast \HOM_{X_1}(s^\ast p_0^\ast f_\ast J, s^\ast p_0^\ast I) \ar[d]^-{\wr}  \\
{p_0}_\ast\HOM_{X_0}({f_0}_\ast q_0^\ast J,p_0^\ast I) \ar[d]^-{\wr} &  {p_0}_\ast s_\ast \HOM_{X_1}({f_1}_\ast u^\ast q_0^\ast J, s^\ast p_0^\ast I) \ar[d]^-{\wr} \\
{p_0}_\ast {f_0}_\ast\HOM_{X_0}(q^\ast_0 J,f_0^! p_0^\ast I) \ar[d]^-{\wr} &  {p_0}_\ast s_\ast {f_1}_\ast\HOM_{X_1}(u^\ast q_0^\ast J, f_1^!s^\ast p_0^\ast I) \ar[d]^-{\wr}  \\
f_\ast {q_0}_\ast\HOM_{X_0}(q^\ast_0 J,f_0^! p_0^\ast I) \ar[r] &  f_\ast {q_0}_\ast u_\ast\HOM_{X_1}(u^\ast q_0^\ast J, u^\ast f_0^! p_0^\ast I) \\
}
\end{displaymath}
In the picture we have applied duality for $f_0,f_1$ but there are no higher derived push-forwards for the two morphisms because both $\HOM_{X_0}(q^\ast_0 J,f_0^! p_0^\ast I)$ and $\HOM_{X_1}(u^\ast q_0^\ast J, f_1^!s^\ast p_0^\ast I)$ are injective.
The morphism $f_\ast {q_0}_\ast\HOM_{X_0}(q^\ast_0 J,f_0^! p_0^\ast I) \to f_\ast {q_0}_\ast u_\ast\HOM_{X_1}(u^\ast q_0^\ast J, u^\ast f_0^! p_0^\ast I) $ in the picture is clearly induced by $\beta$ defined above so that its kernel is $f_\ast\HOM_{\YY}(J,f^{\natural}I)$. This gives us a duality isomorphism:
\begin{displaymath}
  \HOM_{\XX}(f_\ast J,I) \to f_\ast\HOM_{\YY}(J,f^{\natural}I)
\end{displaymath}
The result follows.
\end{proof}
\begin{rem}
  Actually we have proven something stronger than bare duality, we have a sheaf version of the result. We will obtain an analogous sheaf version of the duality for the morphism from a stack to a scheme in the next section.  
\end{rem}

\section{Duality for compact Deligne-Mumford stacks}
\begin{ass}
  From now on schemes and stacks are noetherian separated and of finite type over some algebraically closed field $k$.
\end{ass}
This whole section is devoted to a much more explicit study of duality when stacks involved are Deligne-Mumford. We will retrieve all the classical results like Serre duality and duality for finite morphisms.   

\subsection{Duality and flat base change}
In order to explicitly compute $f^!$ (\textit{twisted inverse image}) in some interesting case we will use existence and a result of flat base-change in the same spirit as \cite[Thm 2]{MR0274464}.

We start proving base change for open immersions. We anticipate a technical lemma which is a variation of \cite[Lem 2]{MR0274464}
\begin{lem}\label{lem:evil-deligne}  
  Let $\XX$ be an algebraic stack and $i\colon\mathcal{U}\to\XX$ an open substack. Let $\mathcal{I}$ be an ideal sheaf defining the complementary of $\mathcal{U}$. For any $\FF\in D(\XX)$ the canonical morphisms:
  \begin{align}
    & \lim_{n\to\infty} \Ext^p_{\XX}(\mathcal{I}^n,\FF) \to H^p(\mathcal{U},i^\ast\FF) \\
    & \lim_{n\to\infty}  \EXT^p_{\XX}(\mathcal{I}^n,\FF) \to R^p i_\ast i^\ast\FF \\ \notag
  \end{align}
are isomorphisms for every $p$.
Moreover if $\GG$ is a bounded above complex over $\XX$ with coherent cohomology we can generalize the first isomorphism to the following:
\begin{equation}
   \lim_{n\to\infty} \Ext^p_{\XX}(\GG \overset{L}{\otimes}\mathcal{I}^n,\FF) \to \Ext_{\mathcal{U}}^p(i^\ast\GG,i^\ast\FF)
\end{equation}
\end{lem}
\begin{proof}
  The statement is the derived version of \cite[App Prop 4]{MR0222093} (we derive $\Hom$ and $\HOM$ using $K$-injectives). This last proposition holds also for stacks.  To see this we can use the usual trick of writing $\Hom_\XX$ as the kernel of an opportune morphism between $\Hom_{X_0}$ and $\Hom_{X_1}$ for some given presentation $X^\bullet$.
\end{proof}
 
Consider the following $2$-cartesian square:
  \begin{equation}\label{eq:unquadrato}
    \xymatrix@1{
      \mathcal{U} \ar[r]^-{i}\ar[d]^-{g} & \XX\ar[d]^-{f} \\
      \mathcal{Z} \ar[r]^-{j} & \YY \\
}
  \end{equation}
where $\XX,\YY\mathcal{Z},\mathcal{U}$ are Deligne-Mumford stacks,  morphisms $f,g$ are proper, $i,j$ are representable flat morphisms.
We can define a canonical morphism  $c_j\colon i^\ast f^!\to g^! j^\ast$. We can actually define it in two equivalent ways according to \cite[pg 401]{MR0274464}. We recall here the two constructions of this morphism for completeness, and to include a small modification that occurs in the case of stacks. 
\begin{enumerate}
\item Since $j$ is flat we have that $j^\ast$ is the left adjoint of $Rj_\ast$ so that we have a unit and a counit:
  \begin{displaymath}
    \phi_j\colon \Id \to Rj_\ast j^\ast \;\quad \psi_j\colon j^\ast Rj_\ast \to \Id
  \end{displaymath}
We can apply a theorem of flat base-change  for the push-forward and obtain an isomorphism: $\sigma\colon j^\ast Rf_\ast\to Rg_\ast i^\ast$ (Proposition $13.1.9$ in \cite{LMBca} states this result without assuming representability of the flat morphism, but I think they are actually using it in the proof). The right adjoint of this gives us $\widetilde{\sigma}\colon Ri_\ast g^!\to f^! Rj_\ast$. The canonical morphism we want is now the composition:
\begin{displaymath}
  i^\ast f^!\xrightarrow{i^\ast f^!\circ\phi_j} i^\ast f^!Rj_\ast j^\ast\xrightarrow{i^\ast\circ\widetilde{\sigma}\circ j^\ast} i^\ast Ri_\ast g^! j^\ast \xrightarrow{\psi_i\circ g^! j^\ast} g^! i^\ast
\end{displaymath}
\item Since $Rf_\ast$ has a right adjoint $f^!$ we have a unit and a counit:
  \begin{displaymath}
    \cotr_f\colon \Id \to f^! Rf_\ast  \;\quad \tr_f\colon Rf_\ast f^! \to \Id
  \end{displaymath}
We can now consider the following composition:
\begin{displaymath}
  i^\ast f^! \xrightarrow{\cotr_g\circ i^\ast f^!} g^!Rg_\ast i^\ast f^! \xrightarrow{g^!\circ\sigma\circ f^!} g^! j^\ast Rf_\ast f^! \xrightarrow{g^! j^\ast \tr_f} g^! j^\ast
\end{displaymath}
\end{enumerate}
According to \cite[pg 401]{MR0274464} both these two compositions define $c_j$; moreover $c_j$ satisfies a cocycle condition when composing two base changes.
\begin{prop}\label{prop:base-change-open}
 In the above setup, assume also that $j$ is an open immersion, the canonical morphism $c_j\colon i^\ast f^!\to g^! j^\ast$ is an isomorphism. 
\end{prop}
\begin{proof}
  Same proof as in \cite[Thm 2, case 1]{MR0274464} but using Lemma \ref{lem:evil-deligne}.
\end{proof}
\begin{cor}
  Duality for proper morphisms of Deligne-Mumford stacks is \'etale local in the sense that given a morphism of stacks $f\colon\XX\to \YY$ and a representable \'etale morphism $j\colon\mathcal{U}\to \YY$ like in picture (\ref{eq:unquadrato}) we have a canonical isomorphism $i^\ast f^!\cong g^!j^\ast$ in $D^+(\YY)$. 
\end{cor}
\begin{proof}
  We apply Zariski main theorem for stacks to $j$ and reduce the problem to the already known open case.
\end{proof}
\begin{rem}
  Unfortunately Zariski main theorem requires representability, even if we have no reasons to suspect that the previous result doesn't hold in the non representable case we don't know a proof.
\end{rem}
\begin{defn}
  We recall that a morphism of schemes (or stacks) $f\colon X\to Y$ is \textit{compactifiable} if it can be written as an open immersion $i$ followed by a proper morphism $p$.
  \begin{displaymath}
    \xymatrix{
      X \ar[d]_f \ar@{^{(}->}[r]^i & \overline{X} \ar[dl]^p \\
      Y & \\
}
  \end{displaymath}
\end{defn}
It is a remarkable result of Nagata (\cite{MR0142549} and \cite{MR0158892}) that every separated finite type morphism of noetherian schemes is compactifiable in this sense. Because of our assumption at the beginning of the section  we can rely on this result of compactification for morphisms of schemes that will appear in the rest of this work.   
Deligne defined in \cite[Appendix]{MR0222093} a notion of ``duality'' for compactifiable morphisms (duality with compact support) of  separated noetherian schemes. First of all we observe that given an open immersion $i$ or more generally an object in a site, the functor $i^\ast$ has a ``left adjoint'' $i_!\colon \text{Coh}_X\to\text{pro-Coh}_{\overline{X}}$ which is an exact functor (see for instance \cite[II Rem 3.18]{Met} for a general enough construction). The category $\text{pro-Coh}$ is the category of pro-objects in the category of coherent sheaves. Given $f$ compactifiable we can define the derived functor $R f_!=R(p_\ast i_!)=(Rp_\ast)i_!$. It is clear that this last functor has a ``right adjoint'' in derived category  that is $i^\ast p^!$ and we will denote it as $f^!$. Deligne proved that this definition of $f^!$ is independent from the chosen compactification and well behaved with respect to composition of morphisms. 
In the previous paragraph we had to put the words \textit{duality} and \textit{adjoint} within inverted commas since there is actually a domain issue. Let $f: X\to Y$ be a separated finite type morphism,  the functor $f^!$ maps objects in $D^+(Y)$ to $D^+(X)$, the functor $f_!$ is defined from $\text{pro-}D^b_{c}(X)$
 to $\text{pro-}D^b_{c}(Y)$ and the duality result in \cite[Appendix Thm 2]{MR0222093} is just a duality result within inverted commas (notation $D_{c}$ means derived category of complexes with coherent cohomology).
We explicitly stress on this subtlety since it has a relevant consequence that is going to affect the rest of this work. When the morphism $f$ is proper we have a duality theory that provides us with a trace morphism $Rf_\ast f^!\to\Id$; given a composition of two proper morphisms $f,g$ and the natural isomorphism between $Rf_\ast Rg_\ast$ and $R (g\circ f)_\ast$, we can deduce, using the trace, a natural isomorphism $\eta_{f,g}$ between $f^!g^!$ and $(g\circ f)^!$; however if the two morphisms are not proper we have no trace morphism on $D^+(X)$ and we cannot deduce a natural isomorphism $\eta_{f,g}$ for a composition. As a side effect it is not clear anymore what kind of uniqueness we could have for the ``dualizing'' functor $f^!$ when  $f$ is not proper. In the proper case the adjoint $f^!$ with trace $\tr_f$ is unique in the sense that if we have another one $f^?$ with trace $\tr_{f}^?$ there is only one isomorphism between them making the two traces compatible. Equivalently we can use the transformation $\tr_{\cdot}$ to produce isomorphisms $\eta_{\cdot,\cdot}$ between compositions, and $\tr_{\cdot}^?$ to produce isomorphisms $\eta^?_{\cdot,\cdot}$, then we can restate the uniqueness result saying that there is only one isomorphism making $\eta_{\cdot,\cdot}$ and $\eta^?_{\cdot,\cdot}$ compatible. 
In the non necessarily proper setup Lipman proved a uniqueness result \cite[Thm 4.8.1]{lipman1960}  that we restate here for completeness:
\begin{thm}[Lipman]\label{thm:lipman-uniqueness}
  Let $f:X\to Y$ be a morphism between finite type separated schemes. There is a unique functor $f^!:D^+(Y)\to D^+(X)$, up to a unique isomorphism, satisfying the three following conditions:
  \begin{enumerate}
  \item when $f$ is proper it restricts to the adjoint of $Rf_\ast$
  \item when $f$ is \'etale (open) it is the pull back
  \item for every fiber square:
    \begin{displaymath}
      \xymatrix{
        \bullet \ar[r]^i \ar[d]^g & \bullet \ar[d]^f \\
        \bullet \ar[r]^j & \bullet \\
      }
    \end{displaymath}
    where $f,g$ are proper and $i,j$ are \'etale (open) the isomorphism:
    \begin{displaymath}    
      i^\ast f^! = i^!f^!\xrightarrow{\eta_{i,f}} (f\circ i)^! = (j\circ g)^! \xleftarrow{\eta_{g,j}} g^! j^\ast
    \end{displaymath}
    is the isomorphism $c_j$ (the one in Proposition \ref{prop:base-change-open})
  \end{enumerate}
\end{thm}
The great relevance of this result for the present work will become evident in the next section.
In order to use this theorem we have now to prescribe isomorphisms $\eta_{\cdot,\cdot}$ in a way that they restrict to the usual thing when the two morphisms are proper, they restrict to the natural isomorphism between pullbacks when the two morphisms are \'etale (open), and they are compatible with the third requirement of the previous theorem. First of all we recall how to compose compactifiable morphisms. Given  $X\xrightarrow{f} Y\xrightarrow{g} Z$ finite type morphisms of noetherian schemes  we can compose them according to the following diagram:
\begin{equation}\label{eq:compat-comp}
  \xymatrix{
    X \ar[r]^-{j}\ar[d]_-{f} & \overline{X} \ar[r]^-{k}\ar[dl]_-{\overline{f}} & \overline{W}\ar[dl]_-{\overline{h}} \\
    Y \ar[r]^-{i}\ar[d]_-{g} & \overline{Y} \ar[dl]_-{\overline{g}} & \\
    Z & & \\
}
\end{equation}
where $i,j,k$ are open immersions, $\overline{f},\overline{g},\overline{h}$ are proper and we can always assume the top right square (delimited by $(\overline{f},\overline{h},i,k)$) to be cartesian; to this purpose we can first apply Nagata's compactification to  the composition $i\circ f$ and choose $\overline{X}$ to be the fibered product $Y\times_{\overline{Y}} \overline{W}$ and $j$ is given by the universal property of the fibered product. The same choice is possible even if we compactify with $i,j,k$  \'etale instead of open; in this case being $k$ \'etale and $k\circ j$ \'etale  the morphism $j$ is also \'etale (\cite[I Cor 3.6]{Met}). We have now a natural way to define an isomorphism $\eta_{f,g}$ that is given by the following composition:
\begin{equation}\label{eq:comp-smooth-shriek}
  (g\circ f)^!=(k\circ j)^\ast(\overline{g}\circ\overline{h})^!\xrightarrow{\eta_{\overline{g},\overline{h}}} (k\circ j)^\ast \overline{h}^! \overline{g}^!\xrightarrow{c_i} j^\ast \overline{f}^! i^\ast \overline{g}^!= f^!g^!
\end{equation}
This provides  a \textit{base-change setup} as defined in \cite[Def 4.8.2]{lipman1960} where distinguished squares are all cartesian and the isomorphism $\beta_\cdot$, associated to a square, is the $c_\cdot$ defined in Theorem \ref{sec:duality-flat-base-change-duality}. The isomorphism $\eta_{f,g}$ clearly satisfies the first two requirements of Theorem \ref{thm:lipman-uniqueness}, to check that even the third requirement is satisfied we only have to apply equation (\ref{eq:comp-smooth-shriek}) to the following diagram:
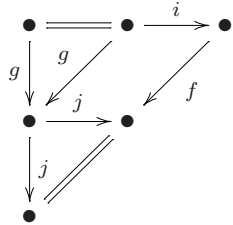
\begin{displaymath}
  \xymatrix{
 \bullet \ar@{=}[r] \ar[d]_g & \bullet \ar[dl]_g \ar[r]^i & \bullet \ar[dl]^f \\
 \bullet \ar[r]^j \ar[d]^j & \bullet \ar@{=}[dl] & \\
 \bullet & & \\
 }
\end{displaymath}
where we have kept notations as in Theorem \ref{thm:lipman-uniqueness}.
\begin{rem}
  We would like to stress on the fact that the previous theorem implies also that the functor $(\cdot)^!$ is independent from the choice of compactification up to a unique isomorphism. Another proof of the independence from the compactification can be found in \cite[Appendix pag 414]{MR0222093}, and in the \'etale case in \cite[VI Thm 3.2.b]{Met}. However Theorem \ref{thm:lipman-uniqueness} is a stronger result than bare ``independence from compactification'' as we will have the opportunity to appreciate in the following section. 
\end{rem}

It is natural to ask how stacks and morphisms of stacks do fit inside this picture of duality for compactifiable morphisms.
We can imagine the following situation: let $\XX$ be a Deligne-Mumford stack and $f_0\colon X_0\to\XX$ an \'etale atlas. The morphism $X_0\to X$ is quasi finite and a fortiori compactifiable. We can choose a special compactification using Zariski main theorem and split the morphism as $X_0\xrightarrow{k}\overline{X}\xrightarrow{h} X$ where $h$ is finite and $k$ is an open immersion. We can also apply Zariski main theorem for stacks \cite[Thm 16.5]{LMBca} to the morphism $f_0$ and obtain a different compactification $X_0\xrightarrow{l} \overline{\XX} \xrightarrow{\rho} \XX \xrightarrow{\pi} X$ where $l$ is open, $\rho$ is finite and $\overline{\XX}$ is a Deligne-Mumford stack. Now we want to prove that these two different compactifications are equivalent from the point of view of duality.
\begin{prop}\label{prop:compatibility-uppershrieck}
  Consider the commutative square:
  
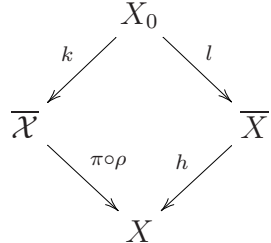
\begin{displaymath}
    \xymatrix{
      & X_0\ar[dr]^-{l} \ar[dl]_-{k} & \\
 \overline{\XX}\ar[dr]^-{\pi\circ\rho} & & \overline{X}\ar[dl]_-{h} \\
     &   X & \\
}
  \end{displaymath}
Let $F\in D^+(X)$, there is a canonical isomorphism $ k^\ast(\pi\circ\rho)^!F\cong l^\ast h^! F$
\end{prop}
\begin{proof}
  Same proof as in \cite[Cor 1]{MR0274464} but using base change result in \ref{prop:base-change-open}. 
\end{proof}
We can now generalize to stacks the independence from the compactification:
\begin{cor}\label{cor:stacky-comp-etale-open}
  In the setup of the previous proposition, for every $F\in D^+(X)$  there is a canonical isomorphism $f_0^\ast\pi^!F \cong l^\ast h^!F$
\end{cor}
\begin{proof}
  We use the previous proposition, the representability of $f_0$ and the analogous result for schemes.
\end{proof}
 We can prove  Grothendieck duality in its sheaf version (on the \'etale site):
 \begin{cor}\label{cor:duality-sheaf-version}
   Let $f\colon \XX\to \YY$ be a proper morphism of Deligne-Mumford stacks and $\FF\in D_c^+(\XX)\text{,}\; G\in D^+(\YY)$. The natural morphism:
   \begin{displaymath}
     \xymatrix{
       Rf_\ast R \HOM_\XX(\FF,f^! G) \ar[r] &  R \HOM_Y(Rf_\ast \FF,Rf_\ast f^! G) \ar[r]^-{\tr_f} & R \HOM_Y(Rf_\ast \FF,G) \\
     }
   \end{displaymath}
is an isomorphism.
 \end{cor}
 \begin{proof}
First of all we observe that according to \cite[15.6.iv]{LMBca} the functor $Rf_\ast$ maps $D^+_c(\XX)$ to $D^+_c(\YY)$ so that everything is well defined.
   Take $I^\bullet,J^\bullet$ $K$-injective complexes quasi isomorphic to $\FF,G$. Let $j\colon U\to \YY$ be an object in the \'etale site of $\YY$ and $X^\bullet$ an \'etale presentation of $\XX$. We can construct the following picture:
   \begin{displaymath}
     \xymatrix{
       U_1 \ar@<-1ex>[r]^-u\ar@<1ex>[r]^-v\ar[d]_-m & U_0 \ar[r]^-l\ar[d]_-k & \mathcal{U} \ar[r]^-g\ar[d]_-j & U \ar[d]_-i \\
       X_1 \ar@<-1ex>[r]^-s\ar@<1ex>[r]^-t  & X_0 \ar[r]^-h & \XX \ar[r]^-f & \YY \\
}
   \end{displaymath}
As usual we have the exact sequence:
\begin{displaymath}
  0\to f_\ast\HOM_{\XX}(I^\bullet,f^!J^\bullet)\to f_\ast h_\ast\HOM_{X_0}(h^\ast I^\bullet,h^\ast f^! J^\bullet)\to f_\ast h_\ast s_\ast \HOM_{X_1}(s^\ast h^\ast I^\bullet,s^\ast h^\ast f^! J^\bullet)
\end{displaymath}
where $f^!J^\bullet$ is a complex of injectives and $\HOM_{\XX}(I^\bullet,f^!J^\bullet)$ is flasque. We first use flat base change to obtain $i^\ast f_\ast h_\ast\HOM_{X_0}(h^\ast I^\bullet,h^\ast f^! J^\bullet)\cong g_\ast l_\ast k^\ast\HOM_{X_0}(h^\ast I^\bullet,h^\ast f^! J^\bullet)\cong $ \\ 
$\cong g_\ast l_\ast \HOM_{U_0}(k^\ast h^\ast I^\bullet,k^\ast h^\ast f^! J^\bullet) $. Using \'etale locality for $f^!$ we obtain $k^\ast h^\ast f^!= l^\ast g^! i^\ast$. Eventually we have:
\begin{displaymath}
  i^\ast f_\ast h_\ast\HOM_{X_0}(h^\ast I^\bullet,h^\ast f^! J^\bullet)\cong  g_\ast l_\ast \HOM_{U_0}( l^\ast j^\ast I^\bullet, l^\ast g^!i^\ast J^\bullet) 
\end{displaymath}
With the same argument we have also:
\begin{displaymath}
  i^\ast f_\ast h_\ast s_\ast \HOM_{X_1}(s^\ast h^\ast I^\bullet,s^\ast h^\ast f^! J^\bullet)\cong  g_\ast l_\ast u_\ast \HOM_{U_1}(u^\ast l^\ast j^\ast I^\bullet, u^\ast l^\ast g^!i^\ast J^\bullet) 
\end{displaymath}
and eventually:
\begin{displaymath}
  i^\ast f_\ast\HOM_{\XX}(I^\bullet,f^!J^\bullet)\cong g_\ast \HOM_{\mathcal{U}}(j^\ast I^\bullet, g^!i^\ast J^\bullet )
\end{displaymath}
Now we just take global sections of this and use the non sheaf version of Grothendieck duality to complete the proof.
 \end{proof}

We have now all the ingredients to prove the flat base change result:
\begin{thm}\label{sec:duality-flat-base-change-duality}
   Consider the following $2$-cartesian square of Deligne-Mumford stacks:
  \begin{displaymath}
    \xymatrix@1{
      \XX' \ar[r]^-{i}\ar[d]^-{g} & \XX\ar[d]^-{f} \\
      \YY' \ar[r]^-{j} & \YY \\
}
  \end{displaymath}
where morphisms $f,g$ are proper and $j$ is flat and representable. The canonical morphism $c_j \colon i^\ast f^!\to g^! j^\ast$ is an isomorphism in $D^+(\YY)$. 
\end{thm}
\begin{proof}
  Same proof as in \cite[Thm 2, case 2]{MR0274464} but using the stacky Corollary \ref{cor:duality-sheaf-version}.
\end{proof}

\subsection{Duality for smooth morphisms}

In this section $\XX\to \Spec{k}$ is a proper Deligne-Mumford stack over an algebraically closed field if not differently specified. 

We start with two local results that don't rely on smoothness.
\begin{lem}
  Let $f\colon \mathcal{Y}\to\mathcal{Z}$ be a representable finite \'etale morphism of noetherian algebraic stacks, then the functor $f^!$ is the same as $f^\ast$.
\end{lem}
\begin{proof}
  If $\mathcal{Y}$ and $\mathcal{Z}$ are two schemes the result is true, then the result for stacks follows from Proposition \ref{prop:representable-proper-duality}.
\end{proof}

Let $\pi\colon\XX \to X$ be a Deligne-Mumford stack with moduli space $X$; using the previous lemma we can study the \'etale local structure of $\pi^!$. The morphism  $\pi$ \'etale locally is the same as  $\rho\colon [\Spec{B}/G]\to\Spec{A}$  where $G$ is a finite group and $\Spec{A}$ is the moduli scheme. Let $\Spec{B}\xrightarrow{p}[\Spec{B}/G]$ be the obvious \'etale (and finite) atlas and $s,t$ source and target in the presentation (both \'etale and finite). Let $I^\bullet$ be an injective complex of $A$-module. Consider the following composition of isomorphisms:
  \begin{displaymath}
    \xymatrix{
      s^\ast(\rho\circ p)^!I^p \ar[r]^{\eta_{s,\rho\circ p}}_-{\widetilde{}} & (\rho\circ p\circ s)^!I^p  & \ar[l]_{\eta_{t,\rho\circ p}}^-{\widetilde{}}  t^\ast(\rho\circ p)^!I^p   \\
    }
  \end{displaymath}
where we have replaced $s^!,t^!$ with $s^\ast,t^\ast$ using the lemma, and every isomorphism is given by equation (\ref{eq:iso-compos-upper-shriek}). Call this isomorphism of injective complexes $\gamma_p$.  The data of $\gamma_\bullet$ and $(\rho\circ p)^! I^\bullet$ define a double complex of injective sheaves on $[\Spec{B}/G]$. We will denote the double complex associated to it with $\rho^\nabla I^\bullet$.
Let $F\in D^+(\Spec{A})$ be  quasi isomorphic to a complex of injectives $I^\bullet$; the injective complex $\rho^\nabla I^\bullet$ on $[\Spec{B}/G]$ defines a functor $\rho^\nabla\colon D^+(\Spec{A})\to D^+([\Spec{B}/G])$
\begin{lem}\label{lem:local-case-nabla}
 The functor $\rho^\nabla$ above is actually $\rho^!$.
\end{lem}
\begin{proof}
  We start observing that the twisted inverse image $(\rho\circ p)^!I$ is just the $B$-module $\Hom_A(B,I)$, the twisted inverse image $(\rho\circ p\circ s)^!I$ is $\Hom_A(B\otimes_A \OO_G,I)$. The natural isomorphism for the composition of twisted inverse images $s^\ast(\rho\circ p)^!I\cong (\rho\circ p\circ s)^!I$ is just the canonical isomorphism $\Hom_A(B,I)\otimes_A\OO_G\xrightarrow{\delta}\Hom_A(B\otimes_A \OO_G,I)$. Let $M$ be a $B$-module with $\alpha$ a coaction of $\OO_G$. From the exact sequence in (\ref{eq:the-exact-sequence}) and using duality  we obtain the following exact diagram:
  \begin{displaymath}
    \xymatrix{
\Hom_A(M\otimes_A\OO_G,I) \ar[r]^{\Hom(\alpha-\iota,\Id)}\ar[d]^{\wr} & \Hom_A(M,I)\ar[ddd]^-{\wr}  \\
\Hom_{B\otimes_A \OO_G}(M\otimes_A\OO_G,\Hom_A(B\otimes_A\OO_G,I))\ar[d]_{\Hom(\Id,\delta^{-1})}^{\wr} &  \\
\Hom_{B\otimes_A \OO_G}(M\otimes_A\OO_G,\Hom_A(B,I)\otimes\OO_G) \ar@{=}[d] &  \\
\Hom_B(M,\Hom_A(B,I))\otimes_A\OO_G \ar[r] & \Hom_B(M,\Hom_A(B,I))  \\
}
  \end{displaymath}
The cokernel of the first horizontal arrow is $\Hom_A(M^G,I)$ while the cokernel of the last horizontal arrow is just $\Hom_B^G(M,\Hom_A(B,I))$ where the coaction of $\OO_G$ on $M$ is $\alpha$ and the coaction on $\Hom_A(B,I)$ is the one of $\rho^\nabla I$. The diagram induces an isomorphism\footnotemark \footnotetext{Despite of the unhappy notation $\Hom_B^G$ is not a $B$-module but a $B^G=A$-module.}:
\begin{displaymath}
  \Hom_A(M^G,I) \to \Hom_B^G(M,\Hom_A(B,I))
\end{displaymath}
By uniqueness of the adjoint we conclude that $\rho^\nabla$ is $\rho^!$.
\end{proof}
\begin{rem}\label{rem:glueing-non-sm}
  Let $F$ be a quasicoherent sheaf on $\spec{A}$; if we know for some reason that $\rho^!F$ is a quasi coherent sheaf itself, then we can conclude that it glues as $\rho^!$ of an injective sheaf in the previous lemma. 
\end{rem}

We can now start implementing the smoothness hypothesis on $\XX$. We recall a result of Verdier in \cite[Thm 3]{MR0274464}:
\begin{thm}\label{thm:verdier-smooth}
Let $f\colon X\to Y$ be a proper morphism of Noetherian schemes and $j\colon U\to X$ an open immersion such that $f\circ j$ is smooth of relative dimension $n$. There exists a canonical isomorphism:
\begin{equation}
  \label{eq:iso-smooth}
  j^\ast f^!\xrightarrow{\lambda_{f\circ j}} \omega_{U/Y}[n]\otimes (f\circ j)^\ast
\end{equation}
where $\omega_{U/Y}$ is the canonical sheaf.
\end{thm}
For the precise definition of $\lambda_{\bullet}$ we redirect the reader to \cite[Thm 3]{MR0274464}.
In order to use Lemma \ref{lem:local-case-nabla}  we need to explicitly know the isomorphism $f^!\circ g^!\xrightarrow{\eta_{f,g}} (g\circ f)^!$ when $f,g$ are  smooth morphisms of schemes and duality is duality with compact support in the sense of theorem \ref{thm:lipman-uniqueness}. 
Unfortunately the Deligne/Verdier approach to duality doesn't suggest anything about the isomorphism $\eta_{f,g}$, but people familiar with duality probably remember that in Grothendieck/Hartshorne's picture the isomorphism $\eta_{f,g}$ is perfectly known when the two morphisms are smooth. If $f,g$ were proper we could use Grothendieck/Hartshorne's result, being confident that the adjoint is unique in the sense explained at the beginning of the previous section; when $f,g$ are not proper we can still use the uniqueness in Theorem \ref{thm:lipman-uniqueness}. 
First of all we recall the classic result in Residues and Duality \cite[III Prop 2.2]{MR0222093} about smooth morphisms.
\begin{lem}[Hartshorne]\label{lem:hartsh-compo-smooth}
  Let $X\xrightarrow{f} Y\xrightarrow{g} Z$ be smooth morphisms of noetherian schemes of relative dimensions $n,m$ respectively. Denote with $f^\sharp$ the functor $\omega_{X/Y}[n]\otimes f^\ast$ (the traditional symbol for smooth duality). There is a an isomorphism of functors:
$\zeta_{f,g}:f^\sharp g^\sharp \to (g\circ f)^\sharp $. The isomorphism $\zeta_{f,g}$ can be chosen to be the isomorphism in \cite[III Def 1.5]{MR0222093}.
\end{lem}
The statement is the same as in Residues and Duality but we have applied the sign fixing proposed by Conrad in \cite{MR1804902}. At this level the isomorphism $\zeta_{f,g}$ appears to be completely arbitrary, nonetheless it becomes intrinsic with theorem \cite[VII 3.4 VAR3]{MR0222093}. We will denote with $(\cdot)^{?}$ the dualizing functor $(\cdot)^!$ of ``Residues and Duality'' \cite[VII 3.4]{MR0222093} and with $\mu_{f,g}\colon f^{?}g^{?}\xrightarrow{} (g\circ f)^{?}$ the isomorphism that controls compositions of functors. Within our assumptions we can always define the functor $(\cdot)^?$:  since every scheme involved is of finite type over a field, according to \cite[V 10]{MR0222093} there is always a dualizing complex.  Let  $e_{f}\colon f^{\sharp}\to f^{?}$ be the isomorphism defined in \cite[3.3.21]{MR1804902}, we can state \cite[VII 3.4 VAR3]{MR0222093} as follows:
\begin{thm}[Hartshorne]
   Let $X\xrightarrow{f} Y\xrightarrow{g} Z$ be smooth morphisms of noetherian schemes, then we have the following compatibility:
   \begin{displaymath}
     \xymatrix{
       f^\sharp g^\sharp \ar[r]^-{\zeta_{f,g}}\ar[d]_{e_f e_g} & (g\circ f)^\sharp\ar[d]_{e_{g\circ f}} \\
   f^? g^? \ar[r]^-{\mu_{f,g}} & (g\circ f)^? \\
}
   \end{displaymath}
\end{thm}
Beside the isomorphism $e_{\cdot}$ is unique in the sense of theorem \cite[VII 3.4]{MR0222093}.

To finally answer our question about the isomorphism $\eta_{f,g}$ we only have to verify that the functor $(\cdot)^?$ satisfies the hypothesis in theorem \ref{thm:lipman-uniqueness}. The first two hypothesis are completely straightforward but the third one does actually require some work. What we need to prove is the following statement:
\begin{prop}\label{prop:ancorauncambiobase}
  Consider the following diagram of schemes:
 \begin{displaymath}
      \xymatrix{
        \bullet \ar[r]^i \ar[d]^g & \bullet \ar[d]^f \\
        \bullet \ar[r]^j & \bullet \\
      }
    \end{displaymath}
where $f,g$ are proper and $i,j$ are open. The following diagram is commutative:
\begin{displaymath}
  \xymatrix{
    i^\ast f^? \ar[rr]^{c_j}\ar[d]_{e_i} & & g^? j^\ast\ar[d]_{e_j} \\
    i^? f^? \ar[r]^{\mu_{i,f}} & (f\circ i)^? & \ar[l]_{\mu_{g,j}} g^? j^? \\    
}
\end{displaymath}
\end{prop}
The statement is very simple when $f,g$ are finite. In this case we can transform the diagram in the following: 
\begin{displaymath}
  \xymatrix{
    i^\ast f^? \ar[rr]^{c_j}\ar@{}[drr] |*+[o][F-]{1} & & g^? j^\ast \\
    i^\ast f^\flat \ar[rr]^{\rho}\ar[d]_{e_i}\ar[u]^{d_f}\ar@{}[drr] |*+[o][F-]{2} & & g^\flat j^\ast\ar[d]_{e_j}\ar[u]^{d_g} \\
    i^? f^? \ar[r]^{\mu_{i,f}} & (f\circ i)^? & \ar[l]_{\mu_{g,j}} g^? j^? \\    
}
\end{displaymath}
where the isomorphism $\rho$ is the isomorphism in \cite[III Prop 6.3 Cor 6.4]{MR0222093}, the symbol $(\cdot)^\flat$ is finite duality and $d_{\cdot}:(\cdot)^\flat \to (\cdot)^?$ is the isomorphism defined in \cite[3.3.19]{MR1804902}. Commutativity of square $1$ is \cite[VII 3.4 VAR6]{MR0222093} and commutativity of square $2$ is \cite[VII 3.4 VAR4]{MR0222093}. Unfortunately it is not enough to use \cite[VII 3.4]{MR0222093} to prove the proposition in general, and as a matter of fact we don't know how to prove this \textit{desired} proposition, even if there are no reasons to believe that it doesn't hold. Luckily we can content ourselves with the finite case. Knowing Proposition \ref{prop:ancorauncambiobase}  for finite morphisms is enough as far as we only use Theorem \ref{thm:lipman-uniqueness}  with compositions of morphisms $\bullet\xrightarrow{f}\bullet\xrightarrow{g}\bullet$ where $g$ is arbitrary but $f$ is quasi-finite. Indeed, if we assume this, we can always split $f$ using Zariski main theorem and reduce to the finite case. This is adequate to our purposes since we will apply Theorem \ref{thm:lipman-uniqueness} only when $f$ is \'etale (the \'etale presentation of a Deligne-Mumford stack).  

We are now ready to solve the local situation for smooth stacks:
\begin{prop}\label{prop:local-smooth-serre}
  Let $\rho\colon \XX=[\Spec{B}/G]\to\Spec{A}$ be an $n$-dimensional smooth Deligne-Mumford stack over $\Spec{k}$ with structure map $\sigma\colon\Spec{A}\to\Spec{k}$ and $p\colon X_0\to\XX$ an \'etale atlas. The dualizing complex  $\rho^!\sigma^! k$ is canonically isomorphic to $\omega_\XX[n]$   
\end{prop}
\begin{proof}
  According to Theorem \ref{thm:verdier-smooth} we have that $(\sigma\circ\rho\circ p)^!k\cong\omega_B[n]$ (it's important to remember that duality along $\sigma$ is duality in the sense of \ref{thm:lipman-uniqueness}), this implies that the dualizing complex is a sheaf (since it is a sheaf after a faithfully flat base change) and according to Remark \ref{rem:glueing-non-sm} it must glue as specified in Lemma \ref{lem:local-case-nabla}. We can now reduce the problem to the commutativity of the following diagram:
  \begin{displaymath}
    \xymatrix{
      s^\ast\rho^!\sigma^!k \ar[r]^{\eta_{s,\rho}}\ar[d]_{s^\ast\eta_{\rho,\sigma}} & (\rho\circ s)^!\sigma^!k\ar[d]_{\eta_{\rho\circ s,\sigma}} & \ar[l]_{\eta_{t,\rho}} t^\ast\rho^!\sigma^!k\ar[d]_{t^\ast\eta_{\rho,\sigma}} \\
      s^\ast(\sigma\circ\rho)^!k \ar[r]^{\eta_{s,\sigma\circ\rho}}\ar[d]_{s^\ast\lambda_{\sigma\circ\rho}} & (\sigma\circ\rho\circ s)^!k\ar[d] & \ar[l]_{\eta_{t,\sigma\circ\rho}} t^\ast(\sigma\circ\rho)^!k\ar[d]_{t^\ast\lambda_{\sigma\circ\rho}} \\
 s^\ast\omega_B[n] \ar[r]^{\zeta_{s,\sigma\circ\rho}} & \omega_{B\times G}[n] & \ar[l]_{\zeta_{t,\sigma\circ\rho}} t^\ast\omega_B[n] \\ 
}
  \end{displaymath}
The line at the top provides the gluing data of the dualizing sheaf according to Lemma \ref{lem:local-case-nabla}, the line at the bottom provides the gluing data of the canonical bundle of $\XX$. The commutativity of the two top squares is granted by the associativity of $\eta_{\cdot,\cdot}$, the commutativity of the two bottom squares is \ref{lem:hartsh-compo-smooth} used together with \ref{thm:lipman-uniqueness} and \ref{prop:ancorauncambiobase}. The canonical isomorphism between the canonical bundle and the dualizing sheaf is given by $\lambda_{\sigma\circ\rho}\circ\eta_{\rho,\sigma}$.
\end{proof}
\begin{rem}
  Suppose now to change the atlas in the previous proposition to some scheme $W_0$ \'etale over $\XX$ but not necessarily finite. We have a new \'etale  presentation $ \xymatrix{ W_1 \ar@<-1ex>[r]^u \ar@<1ex>[r]^v &  W_0\ar[r]^{\tau} & \XX } $. 
Still we have  canonical isomorphisms $u^\ast(\sigma\circ\rho\circ\tau)^!k\cong (\sigma\circ\rho\circ\tau\circ u)^!k$ and the same with $v^\ast$. Recall now that for an \'etale morphism the twisted inverse image is the same as the pullback (whether it is finite or not); according to \ref{lem:hartsh-compo-smooth} and using the previous proposition the two isomorphisms are the two canonical isomorphisms  $u^\ast\omega_{W_0}\cong \omega_{W_1}$ and $v^\ast\omega_{W_0}\cong \omega_{W_1}$. 
\end{rem}
To deal with the global case  we need more informations about the two natural isomorphisms we have: $\eta_{f,g}$ for the composition of twisted inverse images $f^!,g^!$ and $c_j$ for the flat base change by a map $j$ of a twisted inverse image. They are compatible according to a pentagram relation.
\begin{lem} \label{lem:pentagram}
  Consider the following diagram of noetherian schemes where  vertical arrows are flat and the two squares cartesian:
  \begin{displaymath}
\xymatrix{
  X' \ar[r]^g\ar[d]_i & Y'\ar[d]_h \ar[r]^{\rho} & Z' \ar[d]_\sigma \\
   X \ar[r]_f & Y \ar[r]_{\pi} & Z \\
}
  \end{displaymath}
Let $F\in D^+(Z)$, the following pentagram relation holds:
\begin{equation}
 \label{eq:pentagram}
\xymatrix{
    && (\rho\circ g)^!\sigma^\ast F &&  \\
  i^\ast (\pi\circ f)^! F \ar[urr]^-{c_\sigma} & & & &  g^!\rho^!\sigma^\ast F \ar[ull]_-{\eta_{g,\rho}} \\
 & i^\ast f^!\pi^! F \ar[rr]_-{c_h}\ar[ul]^-{i^\ast\eta_{f,\pi}} & & g^!h^\ast\pi^! F\ar[ur]_-{g^!c_\sigma} &
}
\end{equation}
\end{lem}
\begin{proof}
  This is a variation on  \cite[Prop 4.6.8]{lipman1960} using equation (\ref{eq:comp-smooth-shriek}).
\end{proof}
 We are now ready for the global case:
 \begin{thm}[Smooth Serre duality]
   Let $\sigma\colon \XX\to\Spec{k}$ be a smooth proper Deligne-Mumford stack of dimension $n$. The complex $\sigma^!k$ is canonically isomorphic to the complex $\omega_\XX[n]$.
 \end{thm}
 \begin{proof}
   We start with a picture that reproduces the local setup and summarizes all the morphisms we are going to use:
   \begin{equation}\label{picture:presentazione}
     \xymatrix{
Y_1 \ar@<-1ex>[r]^-u \ar@<1ex>[r]^-v \ar@{->>}[d]^-{h_1} & Y_0 \ar[r]^-g\ar@{->>}[d]^-{h_0} & \coprod_{i} [\Spec{B_i}/G_i] \ar[r]^-{\rho}\ar@{->>}[d]^-h &   \coprod \Spec{A_i} \ar[dr]\ar@{->>}[d]^-\sigma & \\
X_1  \ar@<-1ex>[r]^-s \ar@<1ex>[r]^-t & X_0 \ar[r]^-f & \XX \ar[r]^-\pi & X \ar[r] & \Spec{k} \\ 
}
   \end{equation}
We denote with $k_X$ the dualizing complex of the scheme $X$. First we observe that $\pi^!k_X$ is a sheaf. Indeed we have an isomorphism $c_\sigma\colon \rho^!\sigma^\ast k_X\to h^\ast\pi^! k_X$ for Theorem \ref{sec:duality-flat-base-change-duality}; according to Proposition \ref{prop:local-smooth-serre} the complex $\rho^!\sigma^\ast k_X$ is a sheaf and since $h$ is faithfully flat $\pi^! k_X$ must be a sheaf itself. Denote with $\xi$ the isomorphism on the double intersection $X_1$ defining the sheaf $\pi^! k_X$. Using again Proposition \ref{prop:local-smooth-serre} and the remark that follows we have a commutative diagram expressing the isomorphism $\xi$ in relation with the base change isomorphism $c_\sigma$:
\begin{displaymath}
  \xymatrix{
    u^\ast h_0^\ast (\pi\circ f)^! k_X \ar[d]^{u^\ast h_0^\ast \eta_{f,\pi}} \ar[rr]^{u^\ast c_\sigma} &&  u^\ast (\rho\circ g)^! \sigma^\ast k_X \ar[ddr]^-{\eta_{u,\rho\circ g}} \ar[dddd] & \\
  u^\ast h_0^\ast f^\ast\pi^! k_X \ar[dd]^{h_1^\ast\xi}  &&&   \\
   & & & (\rho \circ g\circ u)^! \sigma^\ast k_X \\
  v^\ast h_0^\ast f^\ast \pi^! k_X & & &  \\
   v^\ast h_0^\ast (\pi\circ f)^! k_X \ar[u]_{v^\ast h_0^\ast \eta_{f,\pi}} \ar[rr]^{v^\ast c_\sigma} && v^\ast (\rho\circ g)^! \sigma^\ast k_X \ar[uur]_-{\eta_{v,\rho\circ g}} & \\
}
\end{displaymath}
Now we can use two times the pentagram relation for compactifiable morphisms in (\ref{eq:pentagram}) (remember that $u^!,v^!,t^!,s^!=u^\ast,v^\ast,t^\ast,s^\ast $) and obtain the following commutative diagram:
\begin{equation}\label{picture:double-pentagram}
  \xymatrix{
    u^\ast h_0^\ast (\pi\circ f)^!k_X \ar@{=}[d]\ar[rr]^{u^\ast c_\sigma} && u^\ast (\rho\circ g)^!\sigma^\ast k_X \ar[ddr]^-{\eta_{u,\rho\circ g}} & \\
    h_1^\ast s^\ast (\pi\circ f)^!k_X \ar[dr]^{h_1^\ast \eta_{s,\pi\circ f}} & && \\
    & h_1^\ast (\pi\circ f\circ s)^!k_X \ar[rr]^{c_\sigma} && (\rho \circ g\circ u)^!\sigma^\ast k_X \\
   h_1^\ast t^\ast (\pi\circ f)^!k_X \ar[ur]^{h_1^\ast\eta_{t,\pi\circ f}} &&& \\
  v^\ast h_0^\ast (\pi\circ f)^!k_X \ar[rr]^{v^\ast c_\sigma}\ar@{=}[u] && v^\ast (\rho\circ g)^!\sigma^\ast k_X \ar[uur]_-{\eta_{v,\rho\circ g}} \\
}
\end{equation}
Comparing the two commutative diagrams we have the following commutative square:
\begin{equation}\label{picture:quadrato di merda}
\xymatrix{
 h_1^\ast s^\ast f^\ast \pi^! k_X \ar[d]^{h_1^\ast\xi} \ar[rr]^{h_1^\ast s^\ast\eta_{f,\pi}} && h_1^\ast s^\ast (\pi\circ f)^!k_X\ar[d]^{h_1^\ast (\eta_{t,\pi\circ f}^{-1} \circ \eta_{s,\pi\circ f})} \\
 h_1^\ast t^\ast f^\ast \pi^! k_X \ar[rr]^{h_1^\ast t^\ast\eta_{f,\pi}} && h_1^\ast t^\ast (\pi\circ f)^!k_X \\
}
\end{equation}
First of all we observe that the sheaf $(\pi\circ f)^!k_X$ glued by $\eta_{t,\pi\circ f}^{-1} \circ \eta_{s,\pi\circ f}$ is exactly $\omega_\XX$ according to Lemma  \ref{lem:hartsh-compo-smooth}. 
The commutative square tells us that the isomorphism $\eta_{f,\pi}$ is an isomorphism between the dualizing sheaf and $\omega_\XX$ once it is restricted to $Y_0$; unfortunately $Y_0$ is finer then the atlas we are using ($X_0$) and we actually don't know if this isomorphism descends.  To make it descend we produce a finer presentation of $\XX$, that is we use $Y_0$ as an atlas and we complete the presentation to the groupoid $\xymatrix{\overline{X}_2 \ar@<-1ex>[r] \ar@<0ex>[r] \ar@<1ex>[r] & \overline{X}_1\ar@<-1ex>[r]^{u'} \ar@<1ex>[r]^{v'} & Y_0}$. This gives us an arrow $\lambda$ from $\overline{X}_1$ to $Y_1$ and an arrow from $\overline{X}_2$ to $Y_2$. We can take the square in \ref{picture:quadrato di merda} and pull it back with $\lambda^\ast$ to $\overline{X}_1$. Now the isomorphism $h_0^\ast\eta_{f,\pi}$ descends to an isomorphism of sheaves on $\XX$. The main problem now is that we don't know if $\lambda^\ast h_1^\ast\xi$ and $\lambda^\ast h_1^\ast (\eta_{t,\pi\circ f}^{-1} \circ \eta_{s,\pi\circ f})$ are still the gluing isomorphisms of respectively the dualizing sheaf and $\omega_\XX$. For what concerns the second  we have the following commutative square:
\begin{displaymath}
  \xymatrix{
\lambda^\ast h_1^\ast s^\ast (\pi\circ f)^!k_X \ar[d]_{\lambda^\ast h_1^\ast \eta_{s,\pi\circ f}}\ar[r]^{{u'}^\ast \eta_{h_0,\pi\circ f}} & {u'}^\ast(\pi\circ f\circ h_0)^!k_X \ar[d]^{\eta_{u',\pi\circ f\circ h_0}} \\
\lambda^\ast h_1^\ast (\pi\circ f\circ s)^!k_X \ar[r]^{{v'}^\ast \eta_{h_0,\pi\circ f}} & {v'}^\ast(\pi\circ f\circ h_0)^!k_X 
}
\end{displaymath}
and an analogous one for $t,v'$. This square implies that the sheaf $h_0^\ast(\pi\circ f)^! k_X$ with the gluing isomorphism $\lambda^\ast h_1^\ast \eta_{t,\pi\circ f}^{-1} \circ \eta_{s,\pi\circ f}$ is canonically isomorphic via $h_0^\ast\eta_{h_0,\pi\circ f}$ to the sheaf given by $(\pi\circ f\circ h_0)^!k_X$ and the gluing isomorphism $\eta_{v',\pi\circ f\circ h_0}^{-1}\circ\eta_{u',\pi\circ f\circ h_0}$ which is actually $\omega_\XX$.
For what concerns the dualizing sheaf we start considering the following picture:
\begin{equation}\label{picture:grosso-quadrato-di-merda}
  \xymatrix{
    \overline{X}_1\ar[d]_{\lambda} \ar[r]^{\lambda} & Y_1\ar[d]_{h_1} \ar[r]^v & Y_0\ar[d]^{h_0} \\
    Y_1 \ar[r]^{h_1}\ar[d]_u & X_1 \ar[r]^t\ar[d]_s & X_0 \ar[d]^f \\
    Y_0 \ar[r]^{h_0} & X_0 \ar[r]^f & \XX \\
}
\end{equation}
where every single square is $2$-cartesian. Only the square at the bottom right has a non trivial canonical two-arrow, let's call it $\gamma$. If we think of $\pi^!k_X$ as a sheaf on the \'etale site of $\XX$, the gluing isomorphism $\xi$ on the presentation $\xymatrix{X_1\ar@<-1ex>[r] \ar@<1ex>[r] & X_0}$ is induced by the two-arrow $\gamma$. If we change the presentation to $\xymatrix{\overline{X}_1\ar@<-1ex>[r] \ar@<1ex>[r] & Y_0}$ the gluing isomorphism is induced by $\gamma\ast \Id_{h_1\circ\lambda}$ according to the picture \ref{picture:grosso-quadrato-di-merda}; this implies that the induced isomorphism is exactly $\lambda^\ast h_1^\ast\xi$. We can conclude that the sheaf $h_0^\ast f^\ast\pi^! k_X$ with gluing data $\lambda^\ast h_1^\ast\xi$ is again the dualizing sheaf.
 \end{proof}
Keeping all the notations of the previous theorem we have the following non-smooth result:
 \begin{thm}\label{thm:non-sm-serre}
    Let $\sigma\colon \XX\to\Spec{k}$ be a proper Deligne-Mumford stack. Let $F$ be a quasicoherent sheaf on the moduli space $X$, assume that $\pi^!F$ is a quasicoherent sheaf on $\XX$ then its equivariant structure is given by the isomorphism $\eta_{t,\pi\circ f}^{-1} \circ \eta_{s,\pi\circ f}$.
 \end{thm}
 \begin{proof}
   First of all we observe that $\pi^!F$ being a sheaf can be checked \'etale locally as in the previous theorem, to be more specific it is enough to know that $\rho^!\sigma^\ast F$ is a sheaf. We achieve the result of the theorem repeating the same proof as in the smooth case and keeping in mind Remark \ref{rem:glueing-non-sm}.
 \end{proof}

\subsection{Duality for finite morphisms}

We are going to prove that given $f\colon\XX\to\YY$ a representable finite morphism of Deligne-Mumford stacks the functor $f^!$ is perfectly analogous to the already familiar one in the case of schemes. 

To start with the proof we first need to state a couple of results, well known in the scheme-theoretic set up (\cite[Prop 1.3.1]{MR0163909} and \cite[Prop 1.4.1]{MR0163909}), in the stack-theoretic set-up. 
The first one is taken from \cite[Prop 14.2.4]{LMBca}.
\begin{lem}
  Let $\XX$ be an algebraic stack over a scheme $S$. There is an equivalence of categories between the category of algebraic stacks $\YY$ together with a finite schematically representable $S$-morphism $f\colon\XX\to\YY$ and quasicoherent $\OO_\XX$-algebras. 
This equivalence associates to the stack $\YY$ and the morphism $f$ the sheaf of algebras $f_\ast\OO_\YY$; to a sheaf of algebras $\mathcal{A}$ the affine morphism $f_\mathcal{A}\colon\Spec{\mathcal{A}}\to\XX$. 
\end{lem}
From this lemma we can deduce the following result on quasicoherent sheaves:
\begin{lem}
 Let $\XX$ be as above and $\mathcal{A}$ an $\OO_\XX$-algebra. There is an equivalence of categories between the category of quasicoherent $\mathcal{A}$-modules and the category of quasicoherent sheaves on $\Spec{\mathcal{A}}$. Denoted with $f$ the affine morphism $\Spec{A}\to\XX$, and given $\FF$ a quasicoherent sheaf on $\Spec{A}$, the equivalence associates to $\FF$ the sheaf $\overline{f}_\ast\FF$ that is the sheaf $f_\ast\FF$ with its natural structure of $\mathcal{A}$-algebra. The inverted equivalence is the left-adjoint of $\overline{f}_\ast$, we will denote it with $\overline{f}^\ast$ and it maps the category $\text{QCoh}_{f_\ast\OO_\YY}$ to $\text{QCoh}_{\OO_\YY}$. 
\end{lem}
We need also a couple of properties of the functor $\overline{f}^\ast$:
\begin{lem}
\begin{enumerate}
\item  The functor $\overline{f}^\ast$ is exact.
\item  Let $f\colon\XX\to\YY$ be an affine morphism of algebraic stacks and consider a  base change $p$:
  \begin{displaymath}
    \xymatrix{
      \XX_0 \ar[r]^-q\ar[d]_-{f_0} & \XX\ar[d]^-f \\
      \YY_0 \ar[r]^-p & \YY \\
}
  \end{displaymath}
The following base change rule holds:
\begin{equation}
  \label{eq:base-change-sbirulo}
  \overline{f}^\ast p^\ast\cong q^\ast\overline{f}_0^\ast
\end{equation}
and the isomorphism is canonical.
\end{enumerate}
\end{lem}
\begin{proof}
  See \cite[III.6]{MR0222093} for some more detail.
\end{proof}

Given a finite representable morphism of algebraic stacks $f\colon\YY\to\XX$ we are now able to define the following functor:
\begin{equation}
  \label{eq:f-bemolle}
  f^{\flat}\FF= \overline{f}^\ast R\HOM_\XX(f_\ast\OO_\YY,\FF)= R(\overline{f}^\ast\HOM_{\XX}(f_\ast\OO_\YY,\FF));\quad \FF\in D^+(\XX)
\end{equation}
where the complex $R\HOM_{\XX}(f_\ast\OO_\YY,\FF)$ must be considered as a complex of $f_\ast\OO_{\YY}$-modules.
\begin{thm}[finite duality]\label{thm:dual-finite-morph}
  Let $f\colon\YY\to\XX$ be a finite representable morphism of algebraic stacks. The twisted inverse image $f^!$ is the functor $f^\flat$. 
\end{thm}
\begin{proof}
  Let $I$ be an injective quasicoherent sheaf on $\XX$ defined by the couple $(I_0,\alpha)$ on a presentation of $\XX$ (We keep notations in Proposition \ref{prop:representable-proper-duality}). We start observing that the sheaf $\HOM_\XX(f_\ast\OO_\YY,I)$ is determined, on the same presentation by the following isomorphism:
  \begin{displaymath}
\xymatrix{
    s^\ast\HOM_{X_0}({f_0}_\ast\OO_{Y_0},I_0)\ar[r]^-{b_s}_-{\widetilde{}} & \HOM_{X_1}({f_1}_\ast\OO_{Y_1},s^\ast I_0) \ar[r]^-{\widetilde{\alpha}}_-{\widetilde{}} & \HOM_{X_1}({f_1}_\ast\OO_{Y_1},t^\ast I_0)  \ldots \\ 
\ldots &   t^\ast\HOM_{X_0}({f_0}_\ast\OO_{Y_0},I_0)\ar[l]_-{b_t}^-{\widetilde{}} \\
} 
 \end{displaymath}
where $b_s,b_t$ are the two natural isomorphisms and $\widetilde{\alpha}$ is induced by $\alpha$. Applying $\overline{f}_1^\ast$ to this isomorphism we obtain the one defining $f^{\flat}I$. According to equation (\ref{eq:base-change-sbirulo}) we have $\overline{f}_1^\ast s^\ast\cong u^\ast\overline{f}_0^\ast$ and  $\overline{f}_1^\ast t^\ast\cong v^\ast\overline{f}_0^\ast$. The composed isomorphism:
\begin{displaymath}
  u^\ast\overline{f}_0^\ast\HOM_{X_0}({f_0}_\ast\OO_{Y_0},I_0)\xrightarrow{\text{can}}\overline{f}_1^\ast s^\ast\HOM_{X_0}({f_0}_\ast\OO_{Y_0},I_0) \xrightarrow{\overline{f}_1^\ast b_s} \overline{f}_1^\ast \HOM_{X_1}({f_1}_\ast\OO_{Y_1},s^\ast I_0)
\end{displaymath}
is the same as $c_s$ and we can repeat the argument for $t$. Comparing with the isomorphism called $\beta$ in Proposition \ref{prop:representable-proper-duality} we prove the claim.
\end{proof}

Assuming that $\XX$ is projective we can also reproduce some vanishing result for $R\HOM_{\XX}$ like in Hartshorne \cite[Lem 7.3]{Hag}. First a technical lemma:
\begin{lem}\label{lem:local-global-ext-polariz}
  Let $\XX$ be a projective Deligne-Mumford stack, $\OO_X(1)$ a very-ample line bundle on the moduli scheme $X$ and $\EE$ a generating sheaf. Let $\FF,\GG$ be coherent sheaves on $\XX$. For every integer $i$ There is an integer $q_0>0$ such that for every $q\geq q_0$:
  \begin{displaymath}
    \Ext^i_\XX(\FF,\GG\otimes\EE^\vee\otimes\pi^\ast\OO_X(q))\cong\Gamma(\XX,\EXT_\XX^i(\FF,\GG\otimes\EE^\vee\otimes\pi^\ast\OO_X(q)))
  \end{displaymath}
\end{lem}
\begin{proof}
  Same proof as in  \cite[Prop 6.9]{Hag} with obvious modifications.
\end{proof}
\begin{lem}\label{lem:ext-vanishing-closed}
  Let $\YY$ be a codimension $r$ closed substack in an $n$-dimensional smooth projective stack $\XX$. Then $\EXT_{\XX}^i(\OO_\YY,\omega_\XX)=0$ for all $i< r$.
\end{lem}
\begin{proof}
  The proof goes more or less like in \cite[Lem 7.3]{Hag}. Denote with  $F^i$ the coherent sheaf $\EXT^i_\XX(\OO_\YY,\omega_\XX)$. For $q$ large enough the coherent sheaf $F_{\EE}(F^i)(q)$ is generated by the global sections; if we can prove that $\Gamma(X,F_{\EE}(F^i)(q))=0$ for $q>>0$ we have also that $F_{\EE}(F^i)=0$. In Lemma \cite[Lem 3.4]{stack-stab} we have proven that $\pi\supp{F^i}=\supp{F_{\EE}(F^i)}$, using this result we conclude that if $F_{\EE}(F^i)(q)$ has no global sections for $q$ big enough the sheaf $F^i$ is the zero sheaf. 
We can now study the vanishing of   $\Gamma(X,F_{\EE}(F^i)(q))$:
\begin{displaymath}
  \Gamma(X,F_{\EE}(F^i)(q))=\Gamma(\XX,\EXT_\XX^i(\OO_\YY,\omega_{\XX}\otimes\EE^\vee\otimes\pi^\ast\OO_X(q)))=\Ext^i_{\XX}(\OO_\YY,\omega_{\XX}\otimes\EE^\vee\otimes\pi^\ast\OO_X(q))
\end{displaymath}
The last equality holds for a possibly bigger $q$ according to Lemma \ref{lem:local-global-ext-polariz}. Applying smooth Serre duality we have the following isomorphism:
\begin{displaymath}
  \Ext^i_{\XX}(\OO_\YY,\omega_{\XX}\otimes\EE^\vee\otimes\pi^\ast\OO_X(q))\cong H^{n-i}(\XX,\OO_\YY\otimes \EE\otimes\pi^\ast\OO_X(-q))
\end{displaymath}
This last cohomology group is the same as $H^{n-i}(X,F_{\EE}(\OO_\YY)(-q))$; using again \cite[Lem 3.4]{stack-stab}  we have that the dimension of $F_{\EE}(\OO_\YY)(-q)$ is $n-r$ so that the cohomology group vanishes for $i<r$. 
\end{proof}

\begin{cor}[Serre Duality]\label{cor:serre-duality}
  Let $\XX$ be a proper Deligne-Mumford  stack of pure dimension $n$ and $i\colon\XX\to\mathcal{P}$ a codimension $r$ closed embedding in a smooth proper Deligne-Mumford stack $f\colon\mathcal{P}\to\Spec{k}$. The complex  $\EXT_{\mathcal{P}}^\bullet(\OO_\XX,\omega_{\mathcal{P}})[n]$ is the dualizing complex of $\XX$:
  \begin{displaymath}
    i^!f^!k=\EXT_{\mathcal{P}}^\bullet(\OO_\XX,\omega_{\mathcal{P}})[n]
  \end{displaymath}
if $\XX$ is Cohen-Macaulay the dualizing complex is just the dualizing sheaf $\EXT_{\mathcal{P}}^r(\OO_\XX,\omega_{\mathcal{P}})$.
\end{cor}
\begin{proof}
  We have to prove that $\EXT_{\mathcal{P}}^j(\OO_\XX,\omega_{\mathcal{P}})=0$ for $j\neq r$. We prove that for every point of $\XX$ the stalk of $\EXT_{\mathcal{P}}^q(\OO_\XX,\omega_{\mathcal{P}})$ vanishes for $q\neq r$. We use that for $x$ a point in $\XX$ we have $\EXT_{\mathcal{P}}^j(\OO_\XX,\omega_{\mathcal{P}})_x=\Ext_{\OO_{\mathcal{P},i(x)}}^j (\OO_{\XX,x},\omega_{\mathcal{P},i(x)})$. The stack $\mathcal{P}$ \'etale locally is $[\spec{C}/G]$ where $C$ is regular and $G$ a finite group, since a closed embedding is given by a sheaf of ideals we can assume that $i\colon\XX\to\mathcal{P}$ is given locally  by:
  \begin{displaymath}
    \xymatrix{
      [\spec{B}/G] \ar[r]^i & [\spec{C}/G] \\
    }
  \end{displaymath}
where $B$ is local and Cohen-Macaulay. We denote with $\omega_C$ the canonical sheaf of $[\spec{C}/G]$. According to \cite[Lem 5.3]{MR2312554} we can choose an injective (equivariant) resolution $I^\bullet$ of $\omega_C$ such that the complex:
\begin{displaymath}
  \xymatrix@R=0pt{
    0 \ar[r] & \Hom_C^G(B,I^p) \ar[r] & \Hom_C(B,I^p) \ar[r] & \Hom_C(B,I^p)\otimes G \ar[r] & \\ 
 \ar[r] & \Hom_C(B,I^p)\otimes G\otimes G \ar[r] & \ldots & & \\
}
\end{displaymath}
is exact for every $p$. The arrows are induced by the coaction of $\omega_C$ and of the structure sheaf of $[\spec{B}/G]$. We obtain a double complex spectral sequence $E_1^{p,q}=\Ext_{C}^p(B,\omega_C)\otimes G^{\otimes q}$ abouting to the equivariant $R\Hom_C^G(B,\omega_C)$ (the sheaves $\Hom_C^G(B,I^\bullet)$ are considered as $B^G$-modules) that is the stalk of the curly $\Ext$. Since $B$ is Cohen-Macaulay of the same dimension as $\XX$ we have $\Ext_{C}^p(B,\omega_C)=0$ for $p\neq r$ and the desired result follows. 
\end{proof}

It is important to stress that the previous Corollary \ref{cor:serre-duality} holds for every projective stack in characteristic zero; indeed according to \cite{geomDM} such a stack can be embedded in a smooth proper Deligne-Mumford stack. 
To conclude we prove that $\pi_\ast$ maps the dualizing sheaf of a projective Deligne-Mumford  stack to the dualizing sheaf of its moduli scheme.
\begin{prop}
  Let $\XX$ be an $n$-dimensional tame proper Cohen-Macaulay Deligne-Mumford stack with moduli scheme $\pi\colon\XX\to X$. Denote with $\omega_\XX[n]$ the dualizing sheaf of $\XX$. The quasi coherent sheaf $\pi_\ast\omega_{\XX}[n]=\colon \omega_X[n]$ is the dualizing sheaf of $X$ in $D^b(X)$.
\end{prop}
\begin{proof}
  First of all we observe that the moduli scheme $X$ is projective and Cohen-Macaulay so that we already know that its dualizing complex is actually a sheaf $\omega_X[n]$. We also know that $\pi^!\omega_X[n]=\omega_{\XX}[n]$. We have just to prove that $\pi_\ast\pi^!\omega_X[n]=\omega_X[n]$. Let $F$ be an object in $D^b(X)$, by duality we know that $R\Hom_{X}( F, \pi_\ast \pi^!\omega_X)=R\Hom_{\XX}(L\pi^\ast F, \pi^!\omega_X)= R\Hom_{X}(\pi_\ast L\pi^\ast F, \omega_X)$. We already know that $\pi_\ast \pi^\ast=\Id $ on quasicoherent sheaves and since we are working in the derived category of $\text{QCoh}(X)$ we have also  $\pi_\ast L\pi^\ast = \Id$. Using duality on $X$ and uniqueness of the dualizing sheaf we obtain $\pi_\ast\pi^!\omega_X[n]=\omega_X[n]$. 
\end{proof}
\begin{rem}
  At the moment we cannot extend the result to $D^+(X)$; even if $D_{qc}(X)$ is equivalent to $D(X)$ and the category of $\OO_X$-modules has enough $K$-projectives, we don't know how to define $L\pi^\ast$ from $D_{qc}(X)$ to $D(\XX)$.
\end{rem}
\begin{cor}
  Let $X\to \spec{k}$ be a proper variety with finite quotient singularities. Denote with $\pi\colon \XX^{\text{can}}\to X$ the canonical stack associated to $X$ as in \cite[Rem 4.9]{FMN07} and with $\omega_{\text{can}}$ its canonical bundle. The coherent sheaf $\pi_\ast\omega_{\text{can}}$ is the dualizing sheaf of $X$.  
\end{cor}
\begin{proof}
 We just observe that $\XX^{\text{can}}$ is smooth so that its dualizing sheaf is the canonical bundle and apply the previous corollary.
\end{proof}

To conclude the section we want to discuss what fails when the stack $\XX$ is tame but not Deligne-Mumford. Still we have a theorem that gives us the \'etale local structure of a tame Artin stack as a quotient of affine schemes by linearly reductive group schemes. From this point of view we are in a good situation to reproduce all the arguments so far. However the construction fails when we need arguments related to duality with compact support. For Deligne-Mumford stacks we have first proven that duality is Zariski local. Using Zariski main theorem we have seen that Zariski local implies \'etale local. \'Etale local implies the sheaf version of duality (on the \'etale site) and we can use it to prove that duality is compatible with flat base change.  If we work on the lisse-\'etale site of $\XX$ we miss an argument to prove that duality is smooth local, moreover we have to deal with compositions of morphisms of schemes where the first morphism is smooth but not \'etale and we definitely need to prove Proposition \ref{prop:ancorauncambiobase}.  

\section{Applications and computations}

\subsection{Duality for nodal curves}

Despite being probably already known, it is a good exercise to compute  the dualizing sheaf for a nodal curve using the machinery developed so far. First of all we specify that by nodal curve we mean a non necessarily balanced tame nodal curve over an algebraically closed field $k$ such that each stacky node is mapped to a node in the moduli scheme. This kind of curve is Deligne-Mumford because of the assumption on the characteristic. We can assume from the beginning that the curve has generically trivial stabilizer. If it is not the case, we can always rigidify the curve and treat the gerbe separately. We assume also that if the node is reducible none of the two components has a non trivial generic stabilizer. With this assumption a stacky node, \'etale locally on the moduli scheme, looks like $\bigl[ \spec{\frac{k[x,y]}{(xy)}} / \mu_{a,k}\bigr] $ (or more precisely the localization at the node of this one). The action of $\mu_{a,k}$ is given by:
\begin{displaymath}
\xymatrix@R=0pt{
  \frac{k[k,y]}{(xy)} \ar[r] &  \frac{k[x,y]}{(xy)}\otimes\mu_{a,k} \\
      x,y \ar@{|->}[r] & \lambda^i x, \lambda^j y \\
}
\end{displaymath}
where $(i,a)=1$ and $(j,a)=1$.
The result of this section is the following theorem:
\begin{thm}\label{thm:nodal-duality}
Let $\mathcal{C}$ be a proper tame  nodal curve as specified above. Let $\pi\colon\mathcal{C}\to C$ be its moduli space. Let $D$ be the effective Cartier divisor of $C$ marking the orbifold points, and denote with $\mathcal{D}=\pi^{-1}(D)_{\text{red}}$. Denote also with $\omega_C$ the dualizing sheaf of $C$ and with $\omega_{\mathcal{C}}=\pi^{!}\omega_{C}$ the dualizing sheaf of $\mathcal{C}$. The following relation holds:
\begin{displaymath}
  \omega_{\mathcal{C}}(\mathcal{D})=\pi^\ast\omega_C(D)
\end{displaymath}
\end{thm}

We start proving this theorem with the following local computation:
\begin{lem}\label{lem:loc-duality-node}
  Consider the orbifold node $\mathcal{Y}:= \bigl[ \spec{\frac{k[x,y]}{(xy)}/\mu_{a,k}}\bigr] $ described above. Let $\rho\colon \mathcal{Y}\to Y:=\spec{\frac{k[u,v]}{(uv)}}$ be the moduli scheme, then we have $\omega_{\mathcal{Y}}=\rho^!\OO_{Y}=\OO_{\mathcal{Y}}$.
\end{lem}
\begin{proof}
  We will denote the ring $\frac{k[x,y]}{(xy)}$ with $B$ and $\frac{k[u,v]}{(uv)}$ with $A$. We also choose $\spec{B}$ as atlas for the stack. Let $\alpha,\beta$ be the smallest positive integers such that $i\alpha=0 \mod a,\; j\beta=0\mod a$, then the morphism from the atlas to the moduli scheme is the following:
  \begin{displaymath}
    \xymatrix@R=0pt{
      A \ar[r]^{p_0} & B \\
      u,v \ar@{|->}[r] & x^\alpha,y^\beta \\ 
}
  \end{displaymath}
The dualizing sheaf for $\spec{A}$ is isomorphic to the structure sheaf, so it's enough to compute duality for the structure sheaf denoted  as the free $A$-module $\langle e\rangle$. According to \ref{lem:local-case-nabla} we first need to compute the $B$-module $R\HOM_{A}(B,\langle e\rangle)$. We take the infinite  projective resolution of $B$ as an $A$-module:
\begin{displaymath}
  \xymatrix@R=0pt{
\ldots \ar[r] & A^{\oplus(\alpha+\beta-2)} \ar[r] & A^{\oplus(\alpha+\beta-2)} \ar[r] & A^{\oplus(\alpha+\beta-1)} \ar[r] & B \ar[r] & 0 \\
   & & &   \;\, e_0 \ar@{|->}[r] & 1 & \\
  \ldots & e_l \ar@{|->}[r] & u e_l,\, e_l \ar@{|->}[r] & v e_l,\, e_l\ar@{|->}[r] & x^l & \\
 \ldots  & f_m \ar@{|->}[r] & v f_m,\, f_m \ar@{|->}[r] & u f_m,\, f_m\ar@{|->}[r] & y^m & \\
}
\end{displaymath}
where $1\leq l \leq \alpha-1,\, 1\leq m \leq \beta-1$. We apply the functor $\Hom_A(\cdot,\langle e \rangle)$ and compute cohomology. The complex is obviously acyclic as expected, and $h^0$ is the $A$-module $\bigoplus_{l} (u) e_l^\vee\oplus \bigoplus_{m} (v) f_m^\vee \oplus e_0^\vee$. The $A$-module $h^0$ is naturally a sub-module of $\Hom_A(B,\langle e\rangle)$ and its $B$-module structure is induced by the natural $B$-module structure of this last one. Let $g_l^\vee\in\Hom_A(B,\langle e \rangle)$ the morphism such that $g_l^\vee(x^l)=e$ and zero otherwise, $h_m^\vee$ the morphism such that $h_m^\vee(y^m)=e$ and zero otherwise, and $g_0^\vee$ such that $g_0^\vee(1)=e$. The $B$-module $\Hom_A(B,\langle e\rangle)$ can be written as $\frac{\langle g_{\alpha-1}^\vee,h_{\beta-1}^\vee \rangle}{x^{\alpha-1}g_{\alpha-1}^\vee - y^{\beta-1}h_{\beta-1}^\vee}$. The $B$ module structure of $h^0$ is then given  by:
\begin{displaymath}
  \xymatrix@R=0pt{
  \langle e_0^\vee\rangle \ar@{>->}[r] & \frac{\langle g_{\alpha-1}^\vee,h_{\beta-1}^\vee \rangle}{x^{\alpha-1}g_{\alpha-1}^\vee - y^{\beta-1}h_{\beta-1}^\vee} \\
 e_0^\vee \ar@{|->}[r] & x^{\alpha-1}g_{\alpha-1}^\vee \\
}
\end{displaymath}
Eventually we have $\overline{p}_0^\ast R\Hom_A(B,\langle e\rangle)=\langle e_0^\vee \rangle$. To compute the equivariant structure of $\langle e_0^\vee \rangle$ we follow the recipe in Lemma \ref{lem:local-case-nabla}. We find out that the coaction on $e_0^\vee,e_l^\vee,f_m^\vee$ is as follows:
\begin{displaymath}
  \xymatrix@R=0pt{
     e_0^\vee \ar[r] & e_0^\vee \\
     e_l^\vee \ar[r] & \lambda^{-il} e_l^\vee \\
     f_m^\vee \ar[r] & \lambda^{-jm} e_m^\vee \\
}
\end{displaymath}
so that the equivariant structure of $\langle e_0^\vee \rangle$ is the trivial one and $\rho^!\omega_Y$ is then canonically isomorphic to the structure sheaf.
\end{proof}
\begin{rem}
We can notice that the assumptions on the two integers $i,j$ have never been used in the previous proof, however they are going to be necessary in what follows.   
\end{rem}
With the following proposition we take care of smooth orbifold points.
\begin{prop}\label{prop:duality-radice}
  Let $\XX\to\spec{k}$ be a proper Deligne-Mumford stack that is generically a scheme, let $D=\sum_{i=1}^d D_i$ be a simple normal crossing divisor whose support does not contain any orbifold structure. Let $\mathbf{a}=(a_1,\ldots,a_d)$ positive integers. Denote with $\XX_{\mathbf{a},D}=\sqrt[\mathbf{a}]{D/\XX}\xrightarrow{\tau}\XX$ and with $\mathcal{D}_i=(\tau^{-1}D_i)_{\text{red}}$. For every $\FF$ quasicoherent sheaf on $\XX$ the object $\tau^!\FF\in D(\XX_{\mathbf{a},D})$ is the quasicoherent sheaf $\tau^\ast\FF(\sum_{i=1}^d(a_i-1)\mathcal{D}_i)$.
\end{prop}
\begin{proof}
  Since $\tau$ is a flat morphism we already know that $\tau^!$ maps quasicoherent sheaves to quasicoherent sheaves. The precise statement can be retrieved using some of the computations in \cite[Thm 7.2.1]{AGVgwdms} and the machinery in Theorem \ref{thm:non-sm-serre} (details left to the reader).
\end{proof}
\begin{proof}[proof of Theorem \ref{thm:nodal-duality}]
For the moment we can assume that the curve has no other orbifold points then the nodes and without loss of generality we can assume that there is only one node. First we prove that $\pi^!\omega_C=\pi^\ast\omega_C$, then we can add  smooth orbifold points in a second time applying the root construction; the formula claimed in the theorem follows then from Proposition \ref{prop:duality-radice}. 
First of all we take an \'etale cover $\mathcal{Y}$ of $\mathcal{C}$ in this way:  we choose an \'etale chart of the node that is an orbifold node like the one in Lemma \ref{lem:loc-duality-node} and we complete the cover with a chart that is the curve $\mathcal{C}$ minus the node, denoted with $C_0$. The setup is summarized by the following cartesian square:
\begin{displaymath}
  \xymatrix{
  C_0\coprod [\spec{B}/\mu_a] \ar[r]^-{i\coprod \overline{\sigma}}\ar[d]^-{\Id\coprod p_0} & \mathcal{C}\ar[d]^-{\pi} \\
   C_0\coprod \spec{A} \ar[r]^-{i\coprod \sigma} & C \\
}
\end{displaymath}
where $B=\frac{k[x,y]}{(xy)},\; A=\frac{k[u,v]}{(uv)}$, the map $p_0$ sends $u,v$ to $x^\alpha,y^\beta$, the map $i$ is the inclusion of $C_0$ in $C$, and $\sigma$ is \'etale. We use as an atlas $C_0\coprod \spec{B}$ with the obvious map to $\mathcal{C}$. Completing the presentation we obtain the following groupoid:
\begin{displaymath}
  \xymatrix{
 C_0\coprod \spec{B}\times\mu_a \coprod \spec{B}\times_{\sigma\circ p_0,C} C_0 \ar@<1ex>[r]^-{s} \ar@<-1ex>[r]^-{t} & C_0\coprod \spec{B} \\
}
\end{displaymath}
We can divide it in three pieces: one is the trivial groupoid over $C_0$, the second is $\xymatrix{\spec{B}\times\mu_a\ar@<1ex>[r] \ar@<-1ex>[r] &  \spec{B}}$ where the arrows are action and projection; the last one is $\xymatrix{\spec{B}\times_{\sigma\circ p_0,C} C_0 \ar@<1ex>[r]^-{p_0\circ q_1} \ar@<-1ex>[r]^-{q_2 \overline{p}_0} & C_0\coprod \spec{B}}$, where $q_1$ is projection to $\spec{B}$, $q_2$ is projection to $C_0$ and $\overline{p}_0$ fits inside the following cartesian square:
\begin{equation}\label{picture:quadrato-schifoso}
  \xymatrix{
    \spec{B}\times_{\sigma\circ p_0,C}C_0 \ar[r]^{\overline{p}_0}\ar[d]_{q_1} &  \spec{A}\times_{\sigma,C}C_0 \ar[d]^{q_1} \ar[r]^-{q_2} & C_0\ar[d]^-{i}\\
    \spec{B} \ar[r]^{p_0} & \spec{A} \ar[r]^-{\sigma} & C \\
}
\end{equation}
Now we check if the dualizing sheaf glues like $\pi^\ast\omega_C$ on this presentation and we achieve this using Theorem \ref{thm:non-sm-serre}. The result is trivially true for the first piece of the presentation. For the second piece of the presentation it is implied immediately by Lemma \ref{lem:loc-duality-node}. For what concerns the last piece of presentation we have that $\pi^!\omega_C$ glues with the canonical isomorphism $\overline{p}_0^!q_2^\ast i^\ast\omega_C\cong q_1^\ast p_0^!\sigma^\ast\omega_C$ where the canonical isomorphism comes from the cartesian square in picture (\ref{picture:quadrato-schifoso}). However we have $q_1^\ast p_0^!\sigma^\ast\omega_C= q_1^\ast p_0^!\OO_A\otimes q_1^\ast p_0^\ast\sigma^\ast\omega_C$ and $\overline{p}_0^!\OO_{\spec{A}\times_{\sigma,C}C_0}\otimes \overline{p}_0^\ast q_2^\ast i^\ast\omega_C p_0^!\sigma^\ast\omega_C$ where the canonical isomorphism comes from the cartesian square in picture (\ref{picture:quadrato-schifoso}). However we have $q_1^\ast p_0^!\sigma^\ast\omega_C= q_1^\ast p_0^!\OO_A\otimes q_1^\ast p_0^\ast\sigma^\ast\omega_C$ and $\overline{p}_0^!q_2^\ast i^\ast\omega_C=\overline{p}_0^!\OO_{\spec{A}\times_{\sigma,C}C_0}\otimes \overline{p}_0^\ast q_2^\ast i^\ast\omega_C$. According to Lemma \ref{lem:loc-duality-node} the sheaves $p_0^!\OO_A$ and $\overline{p}_0^!\OO_{\spec{A}\times_{\sigma,C}C_0}$ are respectively equal to $\OO_B$ and $\OO_{\spec{B}\times_{\sigma\circ p_0,C} C_0}$. Eventually the gluing isomorphism is just the identity and we can conclude that the dualizing sheaf is $\pi^\ast\omega_C$.
\end{proof}

\subsection{Other examples of  singular curves}
In the previous section we have seen that nodal curves, balanced or not, have a dualizing sheaf that is an invertible sheaf and carry a trivial representation on the fiber on the node. It is not difficult to find examples of singular curves where the representation on the fiber of the singularity is non trivial. 
What follows is a collection of computations of duality with compact support, performed with the same technique used in Lemma \ref{lem:loc-duality-node}. These examples are mere applications of Lemma \ref{lem:local-case-nabla} and we will be able to retrieve these results with a better technique in the next section.
\begin{ex}{\textbf{(Cusp over a line)}}
  Let $B$ be the cusp $k[x,y]/(y^2-x^3)$ with an action of $\mu_{2,k}$ given by $y\mapsto \lambda y$ and $x\mapsto x$. The moduli scheme of the quotient stack is the affine line $k[x]=:A$ with the morphism:
  \begin{displaymath}
    \xymatrix@R=0pt{
      A \ar[r]^f & B \\
      x \ar@{|->}[r] & x \\
}
  \end{displaymath}
The morphism $f$ is flat and the dualizing sheaf restricted to the atlas is the $B$-module $\Hom_A(B,A)$. As a $B$-module this is just $\langle e_1^\vee\rangle$ where $e_1^\vee(y)=1$ and zero otherwise. The coaction is given by $e_1^\vee\mapsto \lambda^{-1}e_1^\vee$. 
\end{ex}
\begin{ex}{\textbf{(Tac-node over a node)}}\label{ex:tac-node}
  Let $B$ be the tac-node $k[x,y]/(y^2-x^4)$ with an action of $\mu_{2,k}$ given by $y\mapsto y$ and $x\mapsto\lambda x$. The moduli scheme of the quotient stack is the node $k[u,y]/(y^2-u^2)=:A$ with the morphism:
  \begin{displaymath}
    \xymatrix@R=0pt{
      A \ar[r]^f & B \\
      u,y \ar@{|->}[r] & x^2,y \\
}
  \end{displaymath}
The stack is reducible. The morphism $f$ is flat again and the dualizing sheaf is the $B$-module  $\langle e_1^\vee\rangle $ ($e_1^\vee(x)=1$ and otherwise zero) with the coaction $e_1^\vee\mapsto \lambda^{-1}e_1^\vee$.
\end{ex}
These two examples look pretty similar but they are actually of a quite different nature. With a simple computation we obtain that the tac-node is actually a root construction over the node $\sqrt[2]{\OO_A,0/\spec{A}}$. For a root construction we expected that kind of dualizing sheaf from, the already studied, smooth case (Lemma \ref{prop:duality-radice}). The cusp is not a root construction, however it is flat on the moduli scheme anyway, and the dualizing sheaf is the same we have for the root construction.

With the following example we see that the dualizing sheaf can be the structure sheaf for nodes other than $xy=0$.
\begin{ex}{\textbf{(Tac-node over a cusp (an irreducible node))}}
    Let $B$ be the tac-node $k[x,y]/(y^2-x^4)$ with an action of $\mu_{2,k}$ given by $y\mapsto \lambda y$ and $x\mapsto\lambda x$. The moduli scheme of the quotient stack is the cusp $k[u,t]/(t^2-u^3)=:A$ with the morphism:
  \begin{displaymath}
    \xymatrix@R=0pt{
      A \ar[r]^f & B \\
      u,t \ar@{|->}[r] & x^2,xy \\
}
  \end{displaymath}
This one stack is irreducible, none of $y-x^2$ and $y+x^2$ can be a closed substack.
With a computation very similar to the one in Lemma \ref{lem:loc-duality-node} we obtain that the dualizing sheaf is $\langle e_0^\vee \rangle$ ($e_0^\vee(1)=1$ and otherwise zero) with the trivial coaction.
\end{ex}

\subsection{Local complete intersections}
According to Corollary \ref{cor:serre-duality} every Cohen-Macaulay proper Deligne-Mumford stack admits a dualizing quasicoherent sheaf; using Corollary \ref{cor:stacky-comp-etale-open} it's immediate to prove that if the stack also have a Gorenstein atlas then the dualizing sheaf is an invertible sheaf. The aim of this section is to reconstruct the classic duality result for local complete intersections as in \cite[III.7.11]{Hag}.
\begin{thm}\label{thm:loc-complete-intersections}
  Let $\XX$ be a proper Deligne-Mumford stack that has a regular codimension $r$ closed embedding in a smooth proper Deligne-Mumford stack $\mathcal{P}$. Denote with $\mathcal{I}$ the ideal sheaf defining the closed-embedding and with $\omega_\mathcal{P}$ the canonical sheaf of $\mathcal{P}$. The dualizing sheaf of $\XX$ is  $\omega_{\mathcal{P}}\otimes \wedge^r {(\mathcal{I}/\mathcal{I}^2)}^\vee $.
\end{thm}
\begin{proof}
  Our task is to compute $\EXT^r_{\mathcal{P}}(\OO_\XX,\omega_{\mathcal{P}})$. We can produce an \'etale cover of $\mathcal{P}$ so that it is locally $[\spec{C}/G]$ where $C$ is a regular ring and $G$ a finite group. We can assume that the regular closed embedding is locally:
  \begin{displaymath}
    [\spec{B}/G] \xrightarrow{i} [\spec{C}/G]
  \end{displaymath}
where $B$ is defined by $(f_1,\ldots,f_r)$ a regular sequence in $C$. Once we have fixed a basis for the ideal sheaf $\mathcal{I}$, the coaction of $G$ is also determined on that basis. We denote with $\beta_f$ the coaction on the basis $(f_1,\ldots,f_r)$, and it is an $r$-dimensional representation. We denote with $\omega_C$ the canonical sheaf on $\spec{C}$, and it also comes with a coaction that is a one-dimensional representation $\rho_C$ (we are assuming that $\omega_C$ is free). We now take the Koszul resolution $K^\bullet$ of $B$. The coactions of $G$ on $\omega_C$ and $\mathcal{I}$ induce a coaction of $G$ on $K^\bullet$ and a coaction on the complex $\Hom_C(K^\bullet,\omega_C)$. In particular the coaction on $\Hom_C(\wedge^r C^{\oplus r},\omega_C)$ is the representation $\rho_C\otimes \det{\beta_f^\vee}$. Both the induced coaction (denoted with $\gamma^\bullet$) and the trivial coaction are morphisms of complexes:
\begin{displaymath}
\xymatrix{
  0 \ar[r] & \Hom_C^G(K^\bullet,\omega_C) \ar[r] & \Hom_C(K^\bullet,\omega_C) \ar@<-1ex>[r] \ar@<1ex>[r]^-{\gamma^\bullet} & \Hom_C(K^\bullet,\omega_C)\otimes \OO_G \\ 
}
\end{displaymath}
and the equalizer is the $C^G$-module of equivariant morphisms. The cohomology of this first complex computes the global sections of $\EXT_{\mathcal{P}}^\bullet(\OO_\XX, \omega_C)$ restricted to $[\spec{B}/G]$. We can also compute cohomology of the second and third complex and we have arrows between cohomologies induced by both the coaction $\gamma^\bullet$ and the trivial coaction:
\begin{displaymath}
  \xymatrix{
   h^\bullet(\Hom_C(K^\bullet,\omega_C)) \ar@<-1ex>[r] \ar@<1ex>[r]^-{h^\bullet(\gamma^\bullet)} & h^\bullet(\Hom_C(K^\bullet,\omega_C)) \otimes \OO_G \\ 
}
\end{displaymath}
This gives the sheaves $h^\bullet(\Hom_C(K^\bullet,\omega_C))$ an equivariant structure; eventually these sheaves with the equivariant structure are $\EXT_{\mathcal{P}}^\bullet(\OO_\XX, \omega_C)$ restricted to $[\spec{B}/G]$. However, we already know that $h^r(\Hom_C(K^\bullet,\omega_C))=\frac{\omega_C}{(f_1,\ldots,f_r)\omega_C}$ and the others are zero. It is easy to check that the equivariant structure of the non vanishing one is the representation $\rho_C\otimes \det{\beta_f^\vee}$. 
To summarize, the dualizing sheaf restricted to $[\spec{B}/G]$ is isomorphic to $\omega_C\otimes_C B$ where $B$ has the non necessarily trivial coaction $\det{\beta_f^\vee}$.
As in the case of schemes this isomorphism is not canonical. If we change the basis $(f_1,\ldots,f_r)$ to a new one $(g_1,\ldots,g_r)$ where $f_i=\delta_{i,j} g_j$ we produce an automorphism on the Koszul complex $K^\bullet$; in particular the last term of the complex $\wedge^r C^{\oplus r}$ carries an automorphism given by $\det{\delta}$. In the equivariant setup, also the representation $\beta_f$ is affected by a change of basis. In particular the new basis carries a representation $\beta_g$ such that $(a^\ast\delta)\circ \beta_g= \beta_f \circ \delta$, where $a$ denotes the action of $G$ on $\spec{C}$. 
 
As in the case of schemes the sheaf $\wedge^r(\mathcal{I}/\mathcal{I}^2)^\vee$ is trivial on $[\spec{B}/G]$ and the change of basis $(f_1,\ldots,f_r)\mapsto (g_1,\ldots,g_r)$ induces an automorphism on the sheaf that is multiplication by $\det{\delta}^{-1}$. Moreover it is straightforward to check that the equivariant structure of the sheaf is given by the representation $\det{\beta_f^\vee}$.  It is also obvious that the representation changes, after a change of basis, according to the formula  $(a^\ast\delta)\circ \beta'= \beta \circ \delta$. Eventually we can conclude that there is an isomorphism between  $\EXT_{\mathcal{P}}^r(\OO_\XX, \omega_C)\vert_{[\spec{B}/G]}$ and $\omega_{\mathcal{P}}\otimes \wedge^r {(\mathcal{I}/\mathcal{I}^2)}^\vee\vert_{[\spec{B}/G]}$ that doesn't depend on the choice of the basis of $\mathcal{I}$, so to speak a canonical isomorphism. This implies that we can glue these local isomorphisms to obtain a global one.
\end{proof}

In the proof of this theorem we have seen how to compute the sheaf $\EXT^r_{\mathcal{P}}(\OO_\XX,\omega_{\mathcal{P}})$ locally on a stack that is $[\spec{B}/G]$ using the Koszul resolution. Even if the stack is not locally complete intersection but Cohen-Macaulay we can use the same technique to compute $\EXT^r$, replacing the Koszul complex with some other equivariant resolution. This approach is obviously a much faster and reliable technique than the one used in section $2.2$.

\begin{ex}{\textbf{(A non Gorenstein example)}}
Let $B$ be the triple point $k[u,v,t]/(uv-t^2,ut-v^2,vt-u^2)$ with an action of $\mu_{a,k}$ given by $u,v,t\mapsto \lambda u, \lambda v, \lambda t$ and we study duality of the quotient stack. This is non Gorenstein since the  reducible ideal $(uvt)$ is a system of parameters and $B$ is Cohen-Macaulay (we apply the Ubiquity Theorem in \cite{MR0153708}).
We can take the closure  in the weighted projective space $\XX=\mathbb{P}(1,1,1,a)$ with coordinates $u,v,t,z$ and $z$ has degree $a$. We denote the closure with $\mathcal{C}$. It follows from the Euler exact sequence \cite[Lem 3.21]{MR2357682} that the canonical sheaf of $\XX$ is  $\omega_{\XX}=\OO_\XX(-3-a)$ . We can produce the following equivariant resolution of $\OO_\mathcal{C}$:
\begin{displaymath}
  0 \to \OO_\XX(-3)^{\oplus 2}\to \OO_\XX(-2)^{\oplus 3} \to \OO_\XX \to \OO_\mathcal{C}\to 0
\end{displaymath}
From this resolution we can compute $\EXT^2_\XX(\OO_\mathcal{C},\OO_\XX)$ that is the equivariant sheaf:
\begin{displaymath}
\EXT_\XX^2(\OO_\mathcal{C},\OO_\XX) =  \frac{\langle e_1, e_2\rangle}{(t e_1+u e_2, ve_1+te_2,ue_1+ve_2)}
\end{displaymath}
with the  coaction $e_i\mapsto \lambda^{3} e_i$ for $i=1,2$ that is induced by $\OO_\XX(3)^{\oplus 2}$. To complete the computation we take the tensor product with $\OO_\XX(-3-a)$ that changes the coaction to $e_i\mapsto \lambda^{-a} e_i$. This is again a coherent sheaf with trivial action on the fiber on the singularity.
\end{ex}
\begin{ex}{\textbf{(Reducible nodes as local complete intersections)}}
Using theorem \ref{thm:loc-complete-intersections} we can reconsider the problem of duality for nodal curves and provide a very fast solution for a reducible node made of two projective lines. Consider the affine scheme $k[x,y]/(xy)$ with the action of $\mu_{a,k}$ given by $x,y\mapsto \lambda^ix,\lambda^jy$ and we only assume $0\leq i,j < a$ (we are not excluding that the two lines could be gerbes). We can always compactify this stack taking the closure $\mathcal{C}$ in the weighted projective stack $\XX=\mathbb{P}(i,j,a)$  where $x,y,z$ have degree respectively $i,j,a$. The dualizing sheaf of $\XX$ with is $\omega_{\XX}=\OO_\XX(-i-j-a)$. The ideal sheaf defining $\mathcal{C}$ as a closed substack is $\OO_\XX(-i-j)$ and applying theorem \ref{thm:loc-complete-intersections} we obtain that $\omega_{\mathcal{C}}$ is the invertible sheaf $\OO_{\mathcal{C}}(-a)$.  It is immediate to verify that this is compatible with the local description in \ref{lem:loc-duality-node} but it is more tricky to reduce this result to Theorem \ref{thm:nodal-duality}. To compare with the theorem we put the same restrictions on $i,j,a$, that is $(i,a)=1$ and $(j,a)=1$. First of all we observe that the moduli space of $\mathcal{C}$ is the reducible node made of two copies of $\mathbb{P}^1$, we denote it with $C$. If we close it in $\mathbb{P}^2$ we obtain that $\omega_C=\OO_C(-1)$. Compactifying the stacky curve we have possibly added two orbifold smooth points $\mathbb{P}(i),\mathbb{P}(j)$. For this reason the divisor $\OO_{\mathcal{C}}(\mathcal{D})$ in Theorem \ref{thm:nodal-duality} is just $\OO_{\mathcal{C}}(a)$, while the corresponding $\OO_C(D)$ is $\OO_C(1)$ (this is correct if $i,j>1$, if not there is an obvious modification to apply that is left to the reader). Putting together these data and applying the theorem we retrieve $\omega_\mathcal{C}=\OO_\mathcal{C}(-a)$.  
\end{ex}
\begin{ex}{\textbf{(An irreducible balanced node)}}
An irreducible node satisfying conditions in theorem \ref{thm:nodal-duality} can be produced computing a colimit as in \cite[Cor A.0.3]{AGVgwdms}. For instance we can take a curve with two orbifold points: the scheme $A=k[t,x]/(t^3-(x^2-1))$ with an action of $\mu_{3,k}$ given by $t,x\to \lambda t,x$ for $\lambda\in k[\lambda]/(\lambda^3-1)$. The trivial gerbe $\mathcal{B}\mu_{3,k}$ has a closed embedding into each of the  orbifold points. If we denote with $i_1$ the  inclusion in one of the two orbifold points and with $i_2$ the inclusion in the other point,  we can glue the two orbifold points together computing the colimit $\xymatrix{\mathcal{B}\mu_{3,k} \ar@<-1ex>[r]^-{i_1}\ar@<+1ex>[r]^-{i_2} &[\spec{A}/\mu_{3,k}]\ar[r]^-{\pi} & \mathcal{C}}$. The curve $\mathcal{C}$ is a balanced irreducible node\footnotemark \footnotetext{The not balanced node can be obtained composing the inclusion of one of the two points with some arbitrary change of the band of $\mathbb{B}\mu_{3,k}$; the resulting colimit is not a global quotient in this case and we are allowed to suspect that this stack has no generating sheaf.} whose presentation is $[\spec{B}/\mu_{3,k}]$ where $B=k[t,u]/(t^5-u^2+t^2)$ and the action is $t,u\mapsto\lambda t,\lambda u$. It is known in general that a tame balanced nodal curve is a global quotient, this is a consequence of \cite[Thm 3.2.3]{MR2445829} since the theorem states the existence of an ample line bundle. 
We can compactify $\mathcal{C}$ taking the closure inside $\mathbb{P}(1,1,3)$ that is given by the homogeneous polynomial $t^5-z(u^2+t^2)$ where $z$ is the degree $3$ variable. The moduli scheme of the compactification is a nodal cubic in $\mathbb{P}^2$. According to theorem \ref{thm:loc-complete-intersections} the dualizing sheaf is the structure sheaf in perfect agreement with theorem \ref{thm:nodal-duality}. 
\end{ex}
\begin{ex}{\textbf{(Nodes that do not respect Theorem \ref{thm:nodal-duality})}}
Just like the tac-node in  example \ref{ex:tac-node} an ordinary stacky node need not to respect Theorem  \ref{thm:nodal-duality}. This can occur for stacky curves that are nodes \'etale locally on the stack but something different \'etale locally on the moduli scheme. For instance let $\mathcal{C}$ be the curve in $\mathbb{P}(1,4,6)$ defined by the equation $zx^2-y^2$, where $x,y,z$ have respectively degree $1,4,6$. This curve has only one singular point in $(x,y)=(0,0)$ which is an ordinary node and has stabilizer $\mu_{6,k}$. It has no other orbifold points and it is irreducible. This is not included in theorem \ref{thm:nodal-duality} since its moduli scheme is $\mathbb{P}^1$. With the technique used in previous examples we can state that the dualizing sheaf is $\OO_{\mathcal{C}}(-3)$ and it carries a non trivial representation on the fiber on the node. 
\end{ex}
\pagestyle{plain}
\bibliographystyle{amsalpha}
\bibliography{bibliostack}

\end{document}